\documentclass[12pt,reqno]{amsart}
\usepackage{amssymb}
\usepackage{txfonts}
\usepackage{amsfonts}
\usepackage{amsmath}
\usepackage{graphicx}
\usepackage{hyperref}
\usepackage{float}
\usepackage{epstopdf}
\usepackage{color}

\allowdisplaybreaks
\newtheorem{theorem}{Theorem}[section]
\newtheorem{proposition}[theorem]{Proposition}
\newtheorem{lemma}[theorem]{Lemma}
\newtheorem{corollary}[theorem]{Corollary}

\renewcommand{\theequation}{\thesection.\arabic{equation}}

\input{tcilatex}

\def\enddoc{\end{document}}

\def\FRAME#1#2#3#4#5#6#7#8
{
 \begin{figure}[ptbh]
 \begin{center}
 \includegraphics[height=#3]{#7}
 \caption{#5}
 \label{#6}
 \end{center}
 \end{figure}
}

\begin{document}

\title{BGD domains in p.c.f. self-similar sets II: spectral asymptotics for Laplacians}

\pagestyle{plain}

\author {Qingsong Gu}
\address{School of Mathematics\\ Nanjing University, Nanjing, 210093, China}
\email{qingsonggu@nju.edu.cn}
{\author {Hua Qiu}}
\address {School of Mathematics\\ Nanjing University, Nanjing, 210093, China}
\email {huaqiu@nju.edu.cn}

\thanks {Qingsong Gu was supported by the National Natural Science Foundation of
China (Grant No.12101303 and 12171354). Hua Qiu was supported by the National Natural Science Foundation of China, grant 12471087.}

\maketitle
\begin{abstract}
Let $K$ be a p.c.f. self-similar set equipped with a strongly recurrent Dirichlet form. Under a  homogeneity assumption, for an open set $\Omega\subset K$ whose boundary $\partial \Omega$ is a graph-directed self-similar set, we prove that the eigenvalue counting function $\rho^\Omega(x)$  of the Laplacian with Dirichlet or Neumann boundary conditions (Neumann only for connected $\Omega$) has an explicit second term as $x\to +\infty$, beyond the dominant Weyl term. If $\partial\Omega$ has a strong iterated structure, we establish that
\begin{equation*}
\rho^\Omega(x)=\nu(\Omega)G\Big(\frac{\log x}2\Big)x^{\frac{d_S}2}+\kappa(\partial\Omega)G_1\Big(\frac{\log x}2\Big)x^{\frac d2}+o\big(x^{\frac d2}\big),
\end{equation*}
where $G$ and $G_1$ are bounded periodic functions, $\nu$ and $\kappa$ are certain reference measures, and $d_S$ and $d$ are dimension-related parameters.

\end{abstract}

\maketitle
\section{ Introduction}\label{sec 1}
\setcounter{equation}{0}\setcounter{theorem}{0}

Let $\Omega$ be a non-empty bounded open set in $\mathbb R^n$, with boundary $\partial\Omega$. Consider the following eigenvalue problem
\begin{equation*}
\begin{cases}
-\Delta u=\lambda u  &\text{ in }\Omega,\\
\hspace{0.53cm} u=0 &\text{ on }\partial\Omega,
\end{cases}
\end{equation*}
where $\Delta=\sum_{k=1}^n\partial^2/\partial x_k^2$ denotes the Laplace operator with Dirichlet boundary conditions. The value $\lambda$ is said to be an eigenvalue of the problem if there exists a non-zero function $u\in H_0^1(\Omega)$ satisfying $-\Delta u=\lambda u$ in the distributional sense. By classical theory, the spectrum of the above problem is discrete, with the only limit point $+\infty$, and each eigenvalue is a positive, real number with finite multiplicity. We can list them in an increasing order
\begin{equation*}
0<\lambda_1\leq\lambda_2\leq\cdots\leq\lambda_n\leq\cdots\rightarrow+\infty,
\end{equation*}
where each eigenvalue is counted according to its multiplicity.
For $x\in\mathbb R$, denote
\begin{equation*}
\rho(x)=\#\{\lambda\leq x:\ \lambda \text{ is a positive eigenvalue of }-\Delta\}
\end{equation*}
 as the  {\it eigenvalue counting function}.

 The study of the asymptotic behavior of $\rho(x)$ as $x\rightarrow+\infty$ has a long and fruitful history. In 1977, extending Weyl's famous formula for $\rho(x)$, M\'{e}tivier \cite{Me} proved that
 \begin{equation*}
 \rho(x)=(2\pi)^{-n}\kappa_n|\Omega|_nx^{n/2}+o(x^{n/2}),
 \end{equation*}
where $\kappa_n$ is the volume of the unit ball in $\mathbb R^n$, $|\Omega|_n$ denotes the $n$-dimensional Lebesgue measure of $\Omega$. It is natural to wonder whether the formula has a second term.
 The problem is closely related to Kac's famous problem ``{\it Can one hear the shape of a drum?}". Mathematically, can one ``hear" the geometric information of the boundary, for example, the dimension (possibly the boundary is not smooth) and volume, through the spectrum of the Laplacian?

The classical Weyl-Berry's conjecture states that if $\Omega\subset \mathbb R^n$ has a ``fractal" boundary $\partial\Omega$ with Hausdorff dimension $H\in[n-1,n]$, the eigenvalue counting function $\rho(x)$ has the following asymptotic formula as $x\rightarrow+\infty$,
\begin{equation}\label{WBformulaclassical}
\rho(x)=(2\pi)^{-n}\kappa_n|\Omega|_nx^{n/2}-c_{n,H}|\partial\Omega|_{H}x^{H/2}+o(x^{H/2}),
\end{equation}
where $|\partial\Omega|_H$ is the $H$-dimensional Hausdorff measure of $\partial\Omega$ and $c_{n,H}$ is a positive constant depending on $n$ and $H$.

Indeed, as suggested by Brossard and Carmona \cite{BC}, the second term needs to be modified by replacing the Hausdorff dimension $H$ of the boundary $\partial\Omega$ with its Minkowski dimension. This was verified in 1991 by Lapidus in \cite{La}, who obtained an implicit estimate for the second term.
Additionally, for the one-dimensional case, this conjecture was later completely solved by Lapidus and Pomerance \cite{LP}.

In contrast with Kac's problem, {\it what if the drum has a fractal membrane and a fractal boundary?} The theory of Laplacians on fractals is closely related to that of Dirichlet forms and Brownian motions. Since the 1980s, it has emerged as an independent research field. On self-similar sets, the pioneering works include the independent constructions of Brownian motions on the Sierpi\'nski gasket by Goldstein \cite{Go}, Kusuoka \cite{Kus}, and Barlow and Perkins \cite{BP}. The method features the analysis on a sequence of compatible graphs and is extended to post-critically finite (p.c.f.) fractals \cite{K1,K2} by Kigami. The construction of Brownian motions can also be realized on the Sierpi\'nski carpet, a typical non-p.c.f. self-similar set, by Barlow and Bass \cite{BB89}. See \cite{L, Sa, KZ, Metz, BB92, BBKT,CQ} and books \cite{B,K,Str} for further studies of Dirichlet forms on fractals.

Before formulating the eigenvalue problem of fractal Laplacians, let us first make some notational conventions. Let $K$ be a self-similar set, and let $(\mathcal E,\mathcal F)$ be a local regular Dirichlet form on $L^2(K,\mu)$, where $\mu$ is a Radon measure on $K$ with full support.  Denote $\Delta_{\mu}$ as the infinitesimal generator of $(\mathcal E,\mathcal F)$, which is the Laplacian on $K$ associated with $\mu$.
Let $\Omega$ be a non-empty open set in $K$. Denote $(\mathcal E_{\Omega},\mathcal F_{\Omega})$ as the Dirichlet form on $L^2(\Omega,\mu|_{\Omega})$ induced by $(\mathcal E,\mathcal F)$. Write $\mathcal F_{\Omega,0}$ as the closure of $\mathcal F_{\Omega}\cap C_0(\Omega)$ in $\mathcal F_{\Omega}$, where $C_0(\Omega)$ is the space of continuous functions compactly supported in $\Omega$.

Consider the eigenvalue problem of $-\Delta_\mu$ with Dirichlet boundary condition and Neumann boundary condition on $\Omega$:
\begin{equation*}
\begin{cases}
\mathcal E_{\Omega}(u,v)=\lambda\int_{\Omega}uvd\mu, & \text{for any } v\in \mathcal F_{\Omega,0},\\
u\in \mathcal F_{\Omega,0},&
\end{cases}
\end{equation*}
and
\begin{equation*}
\begin{cases}
\mathcal E_{\Omega}(u,v)=\lambda\int_{\Omega}uvd\mu, &\text{for any } v\in \mathcal F_{\Omega},\\
u\in \mathcal F_{\Omega}.
\end{cases}
\end{equation*}
By standard theory, when $-\Delta_\mu$ has compact resolvent, the eigenvalue problem has discrete spectrum with the only limit point $+\infty$, and each eigenvalue is a non-negative real number with finite multiplicity. In what follows, we denote $\rho_D^{\Omega}(x)$ and $\rho_N^{\Omega}(x)$ as the eigenvalue counting functions associated with the Dirichlet and Neumann boundary conditions, respectively.

The eigenvalue problem on fractals has significant difference from that on Euclidean spaces. Let us focus on the situation that $(K,\{F_i\}_{i=1}^N,V_0)$ is a p.c.f. self-similar set
equipped with a self-similar, strongly recurrent Dirichlet form $(\mathcal E,\mathcal F)$ and a self-similar measure $\mu$, where $\{F_i\}_{i=1}^N$ with $N\geq2$ is the iterated function system of $K$ and $V_0$ is the boundary of $K$ consisting of finite many points. Let $(r_1,\ldots,r_N)\in(0,1)^N$ be the energy renormalizing factors of $(\mathcal E,\mathcal F)$, and $(\mu_1,\ldots,\mu_N)\in(0,1)^N$ satisfying $\sum_{i=1}^N\mu_i=1$ be the weights of $\mu$.

First, we look at  a special case that $\Omega=K\setminus V_0$.
In 1993, Kigami and Lapidus \cite{KL} proved that the eigenvalue counting function $\rho^{K\setminus V_0}_*(x)$, where $*$ stands for $D$ (Dirichlet) or $N$ (Neumann), satisfies the estimate that, as $x\rightarrow+\infty$,
 \begin{equation}\label{formulaKandL}
\rho_*^{K\setminus V_0}(x)=\begin{cases}
G\Big(\frac{\log x}{2}\Big)x^{\frac{d_S}{2}}+o\big(x^{\frac{d_S}{2}}\big), & \text{ if $\sum_{i=1}^N\mathbb Z\log\sqrt{r_i\mu_i}$ is a discrete subgroup of $\mathbb R$},\\
Cx^{\frac{d_S}{2}}+o\big(x^{\frac{d_S}{2}}\big), & \text{ otherwise,}
\end{cases}
\end{equation}
 where $G$ is a positive periodic function bounded from above and below away from $0$, with period $T$ being the generator of the additive group $\sum_{i=1}^N\mathbb Z\log\sqrt{r_i\mu_i}$, $C$ is some positive constant, and $d_S$ is the unique solution of
$
\sum_{i=1}^N ({r_i\mu_i})^{d_S/2}=1,
$
called the spectral exponent. The first case is usually termed as the lattice case, and the second case is the non-lattice case.
It is known that if $K$ satisfies the open set condition and we choose $\mu_i=c_i^\alpha$ with $c_i$ being the contraction ratio of $F_i$ and $\alpha$ being the Hausdorff dimension of $K$, and further suppose that $r_i=c_i^{\theta}$ with some $\theta>0$, then $d_S=\frac{2\alpha}{\beta}$, where $\beta=\alpha+\theta$ is called the walk dimension of the Brownian motion on $K$. See \cite{Kum} for nested fractals, \cite{HK} for p.c.f. self-similar sets, and \cite[Section 15]{K12} for general case.

For the lattice case, Kigami later \cite{K98} refined the above formula to obtain a sharp remainder estimate, see Theorem \ref{KigamiJFA} for details. In particular, when each $r_i\mu_i$ equals a common constant, it holds that
\begin{equation*}
\rho_*^{K\setminus V_0}(x)=G\Big(\frac{\log x}{2}\Big)x^{\frac{d_S}{2}}+O(1).
\end{equation*}
This can be interpreted as a second term estimate since it is consistent with the fact that the boundary $V_0$ has dimension zero. Further on the Sierpi{\'n}ski gasket, Strichartz \cite{Str12} showed that the remainder term vanishes for almost all large $x$,
using a spectral decimation method originally developed by Shima and Fukushima \cite{Shi,FS}.

Let us consider another typical open set $\Omega=SG\setminus L$ in the Sierpi\'nski gasket ($SG$) which is generated by removing its bottom line $L$. This domain was initially considered by Owen and Strichartz in \cite{OS} to study the boundary value problem for harmonic functions. Recently, Kigami and Takahashi \cite{KT} obtained the explicit expression of the jump kernel of the trace of the Brownian motion on $SG$ to $L$. In 2019, through a spectral decimation method, the second author \cite{Q} characterized the spectrum on $SG\setminus L$ as consisting of three types of eigenvalues and provided sharp estimates for their associated counting functions. As suggested by numerical experiments, he conjectured that there exists a non-constant bounded $\frac{\log 5}{2}$-periodic function $G_1$, such that as $x\rightarrow+\infty$,
\begin{equation*}
\rho^{SG\setminus L}_D(x)=G\Big(\frac{\log x}2\Big)x^{\frac{\log 3}{\log 5}}+G_1\Big(\frac{\log x}2\Big)x^{\frac{\log 2}{\log 5}}+o\big(x^{\frac{\log 2}{\log 5}}\big),
\end{equation*}
an explicit second term estimate.

\begin{figure}[h]
	\includegraphics[width=6cm]{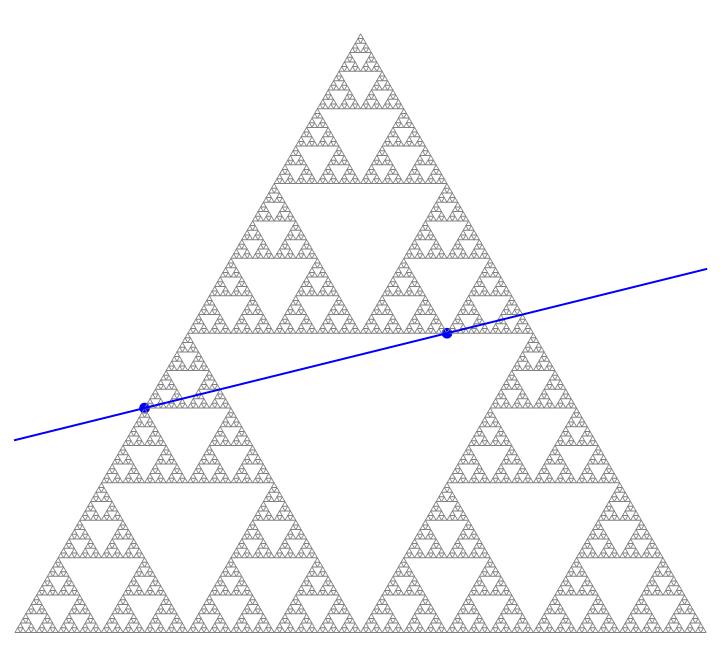}
	\begin{picture}(0,0)
	\end{picture}	
	\caption{domains in $SG$} \label{figureSGcut}
\end{figure}

Recently, for general p.c.f. self-similar sets, the authors of this paper \cite{GQ} introduced the boundary graph-directed condition (BGD) to consider the boundary value problems for harmonic functions on connected open subsets whose geometric boundary are graph-directed self-similar sets. The BGD domains form a broad class of open subsets in p.c.f. self-similar sets. For example, the domains in the Sierpi\'nski gasket generated by cutting $SG$ with a line that passes through two distinct junction points are all BGD domains; see Figure \ref{figureSGcut}. Another typical example is a family of domains in Lindstr{\o}m snowflake whose boundaries are Koch curves; see Figure \ref{figureSF}.

\begin{figure}[h]
	\includegraphics[width=4.5cm]{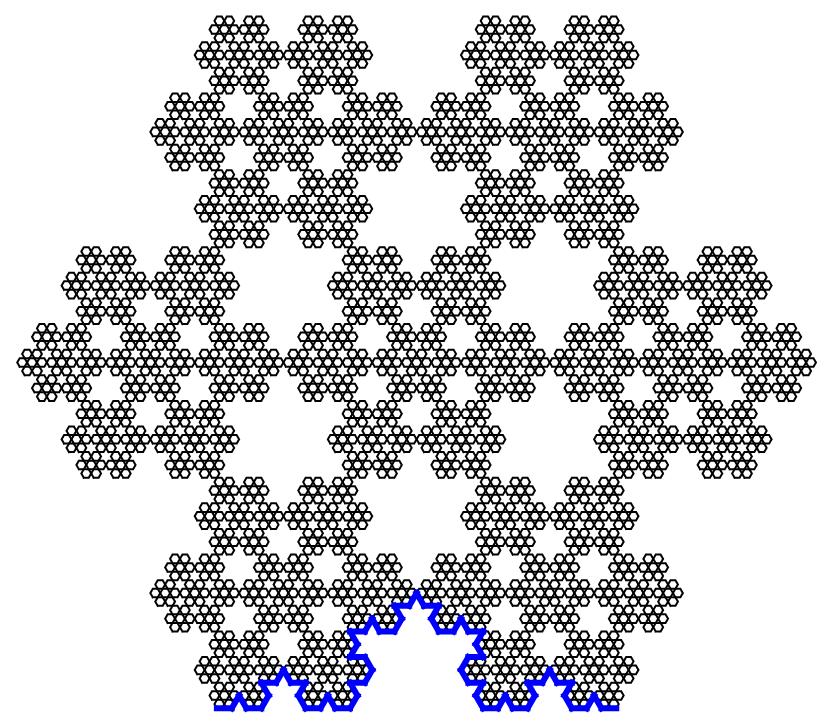}
			\includegraphics[width=4.5cm]{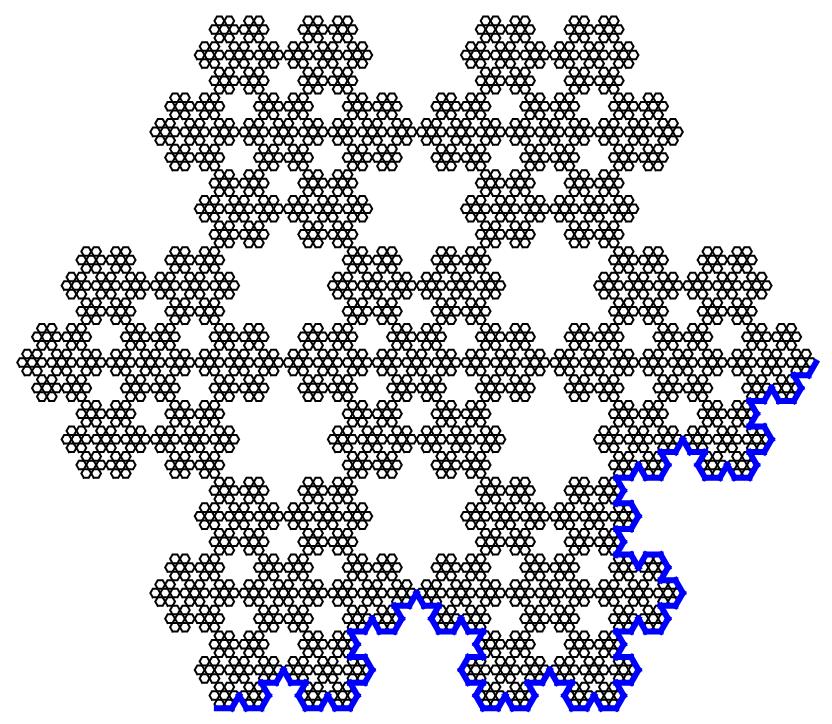}
				\includegraphics[width=4.5cm]{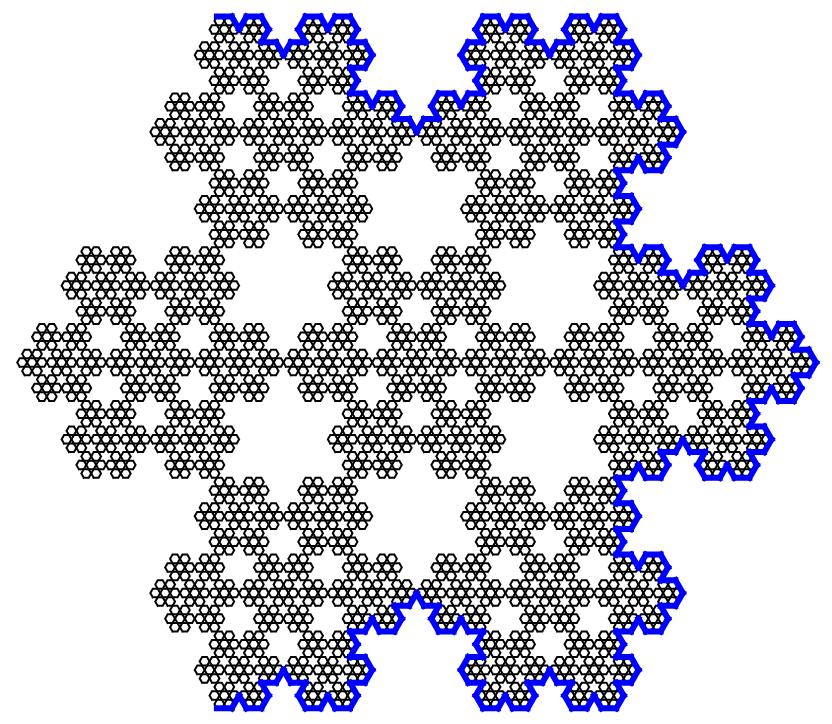}
	\caption{Domains in Lindstr\o m snowflake} \label{figureSF}
\end{figure}

In this paper, for BGD domains in p.c.f. self-similar sets (fractal open sets with fractal boundaries), under certain homogeneity conditions on the Laplacians, we obtain an explicit second term estimate of the asymptotic formula of the eigenvalue counting functions, which can be viewed as a counterpart of \eqref{WBformulaclassical} in the Euclidean case. When the directed graph of BGD domains is strongly connected, we have the following sharp estimates (see Theorem \ref{thm-irreducible}): for $*=D$ or $N$, as $x\rightarrow+\infty$,
\begin{equation*}
\rho^{\Omega}_*(x)=\nu(\Omega)G\Big(\frac{\log x}{2}\Big)x^{\frac{d_S}{2}}+\kappa(\partial\Omega)G_*\Big(\frac{\log x}{2}\Big) x^{\frac{d}2}+o\big(x^{\frac{d}2}\big),
\end{equation*}
where $G$, $G_*$ are bounded periodic functions ($G$, the same function in \eqref{formulaKandL}, depends on $K$, and $G_*$ depends on the shape of $\Omega$), $\nu$ and $\kappa$ are certain reference measures on $\Omega$ and $\partial\Omega$, respectively, reflecting the homogenous structure of the Laplacian under consideration. When the directed graph of BGD domains is not strongly connected, we also have a sharp estimate of the second term, but it might be multiplied by $(\log x)^m$ with some integer $m\geq0$; see Theorem \ref{reducible}.

When $c_i=c$, $\mu_i=\frac1N$, $r_i=r$ for all $i$, in the above asymptotic formula, it is direct to check that
\begin{equation*}
\frac{d_S}2=\frac\alpha\beta,\quad \frac{d}2=\frac{\alpha_{\partial\Omega}}\beta,
\end{equation*}
where $\alpha=\frac{\log N}{-\log c}$ is the Hausdorff dimension of $\Omega$, $\beta=\frac{\log (Nr^{-1})}{-\log c}$ is the walk dimension, and  $\alpha_{\partial\Omega}$ is the Hausdorff (Minkowski) dimension of $\partial\Omega$.
 This is consistent with the Euclidean case, where the walk dimension is always $2$.

At last, we mention that the partition function $Z(t)=\sum_{k=1}^\infty e^{-\lambda_kt}$, the Laplace transform of $\rho(x)$, has better analytic properties than $\rho(x)$ itself. The asymptotic behavior of $Z(t)$ as $t\rightarrow0+$ can be derived from the asymptotic behavior of $\rho(x)$ as $x\rightarrow+\infty$. However, the inverse process is not straightforward.

For the classical Sierpi{\'n}ski carpet ($SC$), Kajino \cite{Ka1}\cite{Ka2} (based on a result of Hambly \cite{Ha}) provided a sharp asymptotic formula for $Z(t)$ of the Laplacian on $SC$:  for $*=D$ or $N$, as $t\rightarrow0+$,
\begin{equation*}
Z(t)=t^{-\frac{\alpha}{\beta}}G_{*,0}(-\log t)+t^{-\frac1{\beta}}G_{*,1}(-\log t)+G_{*,2}(-\log t)+O\left(\exp\left(-ct^{-\frac1{\beta-1}}\right)\right),
\end{equation*}
where $\alpha=\frac{\log8}{\log3}$ is the Hausdorff dimension of $SC$, $\beta$ is the walk dimension of $SC$, and $G_{*,i}$, $i=0,1,2$, are continuous $(\beta\log3)$-periodic functions. Here $c>0$ is  a constant. Kajino's method can also handle the Dirichlet partition function for typical open sets in a p.c.f. self-similar set with good symmetric properties. Thanks to his results, from which we know that the periodic function $G_D$ in the second term of $\rho_D^{SG\setminus L}(x)$ is non-zero. For more details, see Section \ref{subsec6.1}.

\medskip

The paper is organized as follows. In Section \ref{sec 2}, we review some basic concepts related to p.c.f. self-similar sets and Dirichlet forms, as well as key results on spectral asymptotics by Kigami and Lapidus. In Section \ref{sec3}, we discuss the boundary graph-directed condition for open subsets in p.c.f. self-similar sets. In Section \ref{sec 4}, we prove our main result regarding the asymptotic behavior of eigenvalue counting functions in the irreducible case. In Section \ref{sec 5}, we extend this result to the general case. In Section \ref{sec 6}, we provide several examples to illustrate our findings, covering both irreducible and reducible cases. Finally, Section \ref{sec 7} serves as an appendix, presenting some vector-valued renewal theorems that are used in proving Theorems \ref{thm-irreducible} and \ref{reducible}.

\bigskip

\section{Preliminaries}\label{sec 2}
\setcounter{equation}{0}\setcounter{theorem}{0}

\medskip
We begin with some notations about post-critically finite (p.c.f. for short) self-similar sets introduced by Kigami \cite{K}. Let $N\geq 2$ be an integer and $\{F_i\}_{i=1}^N$ be an iterated function system (IFS), i.e. a finite set of contractions, on a complete metric space $(X,d)$. Let $K$ be the associated self-similar set, which is the unique non-empty compact set of $X$ that satisfies the equation
\begin{equation*}
K=\bigcup_{i=1}^N F_i(K).
\end{equation*}
We proceed to define the symbolic space. Let $\Sigma=\{1,\ldots,N\}$ be the {\it alphabets}, and $\Sigma^n$ be the set of {\it words} of length $n$, with $\Sigma^0=\{\emptyset\}$ indicating the set containing only the empty word. The set $\Sigma^\infty$ represents the collection of {\it infinite words} $\omega=\omega_1\omega_2\cdots$. For a word $\omega=\omega_1\cdots\omega_n\in\Sigma^n$, we define its length as $|\omega|=n$, and write $F_{\omega}=F_{\omega_1}\circ \cdots\circ F_{\omega_n}$ the composition of functions (with $F_\emptyset=\rm{Id}$, the identity function). We refer to $F_\omega(K)$ as an {\it $n$-cell}  and denote it as $K_\omega$. Let $\pi: \Sigma^\infty\rightarrow K$ be defined by
\begin{equation*}
\{x\}=\{\pi(\omega)\}=\bigcap_{n\geq1}F_{[\omega]_n}(K),
\end{equation*}
the symbolic representation of $x\in K$ by the word $\omega$, where $[\omega]_n=\omega_1\cdots\omega_n$.

In accordance with \cite{K}, we define the {\it critical set} $\mathcal C$ and {\it post-critical set} $\mathcal P$ for $K$ as follows:
\begin{equation*}
\mathcal C=\pi^{-1}\left(\bigcup_{1\leq i<j\leq N}\left(F_i(K)\cap F_j(K)\right)\right),\quad \mathcal P=\bigcup_{m\geq1}\sigma^m(\mathcal C),
\end{equation*}
where  $\sigma:\ \Sigma^\infty\rightarrow\Sigma^\infty$ denotes the left shift operator, defined by $\sigma(\omega_1\omega_2\cdots)=\omega_2\omega_3\cdots$. If $\mathcal P$ is finite, we call $\{F_i\}_{i=1}^N$ a {\it post-critically finite (p.c.f.)} IFS, and $K$ a {\it p.c.f. self-similar set}. The boundary of $K$ is defined by $V_0=\pi(\mathcal P)$. We also inductively denote
\begin{equation*}
V_n=\bigcup_{i\in\Sigma}F_i(V_{n-1}),\quad V_*=\bigcup_{n=0}^\infty V_n.
\end{equation*}
It is known that the metric space $(K,d)$ has a fundamental neighborhood system $\{K_{n,x}:\ n\geq0,x\in K\}$, where each $K_{n,x}=\bigcup\limits_{\omega\in \Sigma^n: x\in F_\omega(K)}F_{\omega}(K)$, see \cite[Proposition 1.3.6]{K}.
We always assume that $(K,d)$ is connected so that $V_0$ is non-empty. It is clear that $\{V_n\}_{n\geq0}$ forms an increasing sequence of sets, and $K$ is the closure of $V_*$.

\medskip

Our basic assumption on a p.c.f. self-similar set $K$ is the existence of a {\it regular harmonic structure} $(D,{\bf r})$. Denote $Q=\# V_0$. Let ${\bf r}=(r_1,\ldots, r_N)\in (0,\infty)^N$ and $D=(D_{pq})_{p,q\in V_0}$ be a $Q\times Q$ real symmetric matrix satisfying:

\medskip

1. for $u\in\ell (V_0)$, $Du=0$ if and only if $u$ is a constant function;

2. $D_{pq}\geq0$ for any $p,q\in V_0$ with $p\neq q$.

\medskip

For a function $u\in\ell (V_0)$, we define the energy functional $E_0[u]$ as:
\begin{equation*}
E_0[u]=-\sum_{p,q\in V_0}D_{p,q}u(p)u(q),
\end{equation*}
 and for $n\geq1$, we recursively define the energy functional $E_n$ as:
 \begin{equation*}
 E_n[u]=\sum_{\omega\in \Sigma^n}\frac{1}{r_\omega}E_0[u\circ F_{\omega}|_{V_0}] \quad  \text{for } u\in\ell(V_n),
 \end{equation*}
 where $r_\omega=r_{\omega_1}\cdots r_{\omega_n}$ for $\omega=\omega_1\cdots\omega_n$ (with $r_{\emptyset}=1$).

We say that $(D,{\bf r})$ is a harmonic structure on $(K,\{F_i\}_{i=1}^N,V_0)$ if it satisfies the following compatibility condition:
 \begin{equation*}
 E_0[u]=\inf_{v\in \ell(V_1),v|_{V_0}=u} E_1[v]\quad \text{ for } u\in\ell(V_0).
 \end{equation*}
 Furthermore, if ${\bf r}\in (0,1)^N$, we refer to the harmonic structure as {\it regular}.
Under this condition, $ E_n[u]$ forms an increasing sequence with respect to $n$. Consequently, for $u\in C(K)$, the space of all continuous functions on $K$, we can define its energy $\mathcal E[u]$ as:
 \begin{equation*}
 \mathcal E[u]=\lim_{n\rightarrow+\infty} E_n[u|_{V_n}].
 \end{equation*}
Let $\mathcal F=\{u\in C(K):\ \mathcal E[u]<\infty \}$, and define
 \begin{equation*}
 \mathcal E(u,v)=\frac14\big(\mathcal E[u+v]-\mathcal E[u-v]\big) \quad\text{for $u,v\in \mathcal F$.}
 \end{equation*}
Note that by the standard theory \cite{K}, $\mathcal F$ is dense in $C(K)$.

This defines a strongly recurrent self-similar {\it resistance form} $(\mathcal E,\mathcal F)$ satisfying
\begin{equation}\label{ess}
\mathcal E(u,v)=\sum_{i=1}^N\frac{1}{r_i}\mathcal E(u\circ F_i,v\circ F_i)\quad \text{for $u,v\in \mathcal F$},
\end{equation}
where $0<r_i<1$ for $i=1,\ldots,N$ are termed \textit{energy renormalizing factors}.  By iterating \eqref{ess}, it follows that for any $n\geq1$,
\begin{equation}\label{essn}
\mathcal E(u,v)=\sum_{|\omega|=n}\frac{1}{r_\omega}\mathcal E(u\circ F_\omega,v\circ F_\omega)\quad \text{for $u,v\in \mathcal F$}.
\end{equation}
For $u\in\mathcal F$ and $\omega\in\Sigma^n$ for some $n\geq0$, we refer to $\frac1{r_\omega}\mathcal E[u\circ F_{\omega}]$ as the energy of $u$ on the cell $K_{\omega}$.

To define a Laplace operator through the resistance form, we require a measure on the fractal. Let us assume that $\mu$ is a Radon measure with full support on $K$. Then $\mathcal F$ is dense in $L^2(K, \mu)$ and is complete with respect to the $\mathcal E_1^{1/2}$-norm, thus making $(\mathcal E, \mathcal F)$ a {\it Dirichlet form} on $L^2(K, \mu)$, where
\begin{equation*}
\mathcal E_1 [u] = \mathcal E [u] + \int_K u^2d\mu \quad \text{for $u\in \mathcal F$} .
\end{equation*}

\medskip

We set the measure $\mu$ to be a {\it self-similar measure} on $K$. Specifically, we assume that $\mu_1,\mu_2,\cdots,\mu_N$ are positive numbers satisfying $\sum_{i=1}^N\mu_i=1$, then we require $\mu$ to be a probability measure on $K$ such that for any Borel set $A\subset K$,
\begin{equation*}
\mu(A)=\sum_{i=1}^N\mu_i\mu\circ F_i^{-1}(A).
\end{equation*}
Note that $\mu(K_\omega)=\mu_{\omega_1}\mu_{\omega_2}\cdots\mu_{\omega_n}$ for any $\omega=\omega_1\omega_2\cdots\omega_n\in \Sigma^n$.

\medskip

The Laplace operator $\Delta_D$ (or $\Delta_N$) with Dirichlet (or Neumann) boundary condition is defined through the Dirichlet form $(\mathcal E, \mathcal F)$ using weak formulations.
We define $\mathcal F_0=\{f\in\mathcal F:\ f|_{V_0}=0\}$, and write
\begin{equation*}
-\Delta_D u=f \quad\text{for } f\in C(K),
\end{equation*}
if $u\in \mathcal F_0$ satisfies
\begin{equation*}
\mathcal E(u,v)=\int_K fv d\mu \quad \text{for any } v\in\mathcal F_0;
\end{equation*}
write
\begin{equation*}
-\Delta_N u=f \quad\text{for }f\in C(K),
\end{equation*}
if $u\in\mathcal F$ satisfies
\begin{equation*}
\mathcal E(u,v)=\int_K fvd\mu\quad \text{for any }v\in\mathcal F.
\end{equation*}

For $*=D$ or $N$, it is known that the operator $-\Delta_*$ is self-adjoint and possesses a compact resolvent. Say a number $\lambda$ is a  {\it $*$-eigenvalue} of $-\Delta_*$, if there exists a non-zero function $u$ such that
\begin{equation*}
-\Delta_* u=\lambda u.
\end{equation*}
Call the function $u$ a $*$-eigenfunction of $-\Delta_*$ corresponding to $\lambda$.

By a standard theory, the eigenvalues of $-\Delta_*$ are non-negative real numbers, have finite multiplicity, and have $+\infty$ as their sole limit point. Consequently, we can define the associated {\it eigenvalue counting function} on $[0,+\infty)$ as:
\begin{equation}\label{eqecfK}
\rho_*(x)=\#\{k:\ \text{$k\leq x$ and $k$ is a positive eigenvalue of $-\Delta_*$}\},
\end{equation}
where each eigenvalue is counted according to its multiplicity.

Denote $\gamma_i=\sqrt{r_i\mu_i}$ for $i=1,\ldots,N$. Kigami and Lapidus proved:
\begin{theorem}[Kigami-Lapidus \cite{KL}]\label{thm-KL}Let $d_S>0$ be the number such that $\sum_{i=1}^N\gamma_i^{d_S}=1$.

1. Non-lattice case: if the additive group $\sum_{i=1}^N\mathbb Z\log\gamma_i$ is dense in $\mathbb R$, then there exists a constant $C>0$ such that as ${x\rightarrow+\infty}$,
\begin{equation*}
\rho_*(x)=Cx^{\frac{d_S}{2}}+o\big(x^{\frac{d_S}{2}}\big)\quad \text{for $*=D$ or $N$}.
\end{equation*}

2. Lattice case: if the additive group $\sum_{i=1}^N\mathbb Z\log\gamma_i$ is discrete, let $T>0$ be its generator, then there exists a positive (bounded away from $0$), bounded, right-continuous, $T$-periodic function $G$ such that as ${x\rightarrow+\infty}$,
\begin{equation*}
\rho_*(x)=G\Big(\frac{\log x}{2}\Big)x^{\frac{d_S}{2}}+o\big(x^{\frac{d_S}{2}}\big)\quad \text{for $*=D$ or $N$}.
\end{equation*}
\end{theorem}

In a subsequent paper, Kigami refined the remainder term in the lattice case as follows.
\begin{theorem}[Kigami \cite{K98}]\label{KigamiJFA}Under the assumptions of the lattice case in Theorem \ref{thm-KL}, define $Q(z)=(1-\sum_{i=1}^N(z/p)^{m_i})/(1-z)$, where $p=e^{d_ST}$ and $m_i=-\frac{\log\gamma_i}T$ for $i=1,\ldots,N$. Let $\beta=\min\{|z|:\ Q(z)=0\}$ and $m=\max\{ \text{ multiplicity of $Q(z)=0$ at $w$}:\ |w|=\beta,Q(w)=0\}$. Then for $*=D$ or $N$, as $x\rightarrow+\infty$,
\begin{equation}\label{Kigamilattice}
\rho_*(x)=G\Big(\frac{\log x}{2}\Big)x^{\frac{d_S}{2}}+\begin{cases}
O\Big(x^{\frac{d_S}{2}-\frac{\log\beta}{2T}}(\log x)^{m-1}\Big)\qquad &\text{if }p>\beta,\\
O\left((\log x)^{m}\right)\qquad &\text{if }p=\beta,\\
O(1)\qquad &\text{if }p<\beta.\\
\end{cases}
\end{equation}
\end{theorem}

\medskip

Note that if in particular $\gamma_1=\cdots=\gamma_N$, then $Q(z)\equiv1$ and $\beta=+\infty$. Consequently, $p<\beta$ is always satisfied, and the third case in \eqref{Kigamilattice} always holds.

In the rest of this paper, we will focus on the asymptotic behavior of the eigenvalue counting function $\rho_*(x)$ for Laplacians on open subsets of a p.c.f. self-similar set $K$. From Theorem \ref{KigamiJFA} (specifically, the first and second formulas in \eqref{Kigamilattice}), we observe that the ``inhomogeneity" of the scaling factors $r_i\mu_i$ influences the second-order term.  To investigate the impact of the domain boundary on the second-order term, therefore, we only consider the third case of Theorem \ref{KigamiJFA}. For simplicity, we always assume that $\gamma_i=\gamma$ for all $1\leq i\leq N$. Consequently, $T=-\log\gamma$.


\section{Boundary graph-directed condition}\label{sec3}
\setcounter{equation}{0}\setcounter{theorem}{0}

In this section, for a p.c.f. self-similar set $K$, we review the {\it boundary graph-directed condition} (BGD) for an open subset $\Omega$ in $K$, roughly saying that the boundary of $\Omega$ is a graph-directed self-similar set. This condition is initially introduced by the authors in \cite{GQ} for the investigation of boundary value problems for harmonic functions, and will be concerned throughout the paper.

Recall that
graph-directed self-similar sets are an extension of the concept of self-similar sets. Let $(X,d)$ be a complete metric space. Let $(\mathcal A,\Gamma)$ be a {\it directed graph} (permitting loops and multiple edges) with a finite set of {\it vertices} $\mathcal A=\{1,\ldots,P\}$ and a finite set of {\it directed edges} $\Gamma$. For any $\eta\in\Gamma$, if $\eta$ is a directed edge from $i$ to $j$ for some $i,j\in\mathcal A$, we define $I(\eta)=i$ and $T(\eta)=j$, referring to them as the {\it initial vertex} and {\it terminal vertex} of $\eta$, respectively. For $i,j\in\mathcal A$, let $\Gamma(i)=\{\eta\in\Gamma:\ I(\eta)=i\}$ and $\Gamma({i,j})=\{\eta\in\Gamma: I(\eta)=i,\ T(\eta)=j\}$. We assume that each $\Gamma(i)$ is non-empty, and each  edge $\eta$ is associated with a contraction $\Phi_{\eta}$ on $(X,d)$. Then there exists a unique collection of non-empty compact sets $\{D_i\}_{i=1}^{P}$ in $(X,d)$, termed \textit{graph-directed self-similar sets} \cite{MW}, satisfying the equation
\begin{equation}\label{eqGD}
D_i=\bigcup_{\eta\in\Gamma(i)}\Phi_{\eta}(D_{T(\eta)})\quad \text{for  }1\leq i\leq P.
\end{equation}

Let $m\geq1$. A finite word ${\bf\eta}=\eta_1\eta_2\cdots\eta_m$ with $\eta_i\in\Gamma$ for $ i=1,\ldots,m$ is called {\it admissible} if $T(\eta_i)=I(\eta_{i+1})$ for all $i=1,\ldots,m-1$. We define the length of ${\bf\eta}$ as $|{\bf\eta}|=m$, and write $I({\bf\eta})=I(\eta_1)$ and $T({\bf\eta})=T(\eta_m)$. Additionally, we define $\Phi_{\bf\eta}=\Phi_{\eta_1}\circ\cdots\circ\Phi_{\eta_m}$, the composition of contractions. The set of all admissible words of length $m$ is denoted by $\Gamma_m$, and by convention, $\Gamma_0=\{\emptyset\}$ contains only the empty word. For $0\leq n\leq m$, we denote the $n$-th step truncation of ${\bf\eta}$ as $[{\bf\eta}]_n=\eta_1\cdots\eta_n$.  For $i\in\mathcal A$, we also define $\Gamma_m(i)=\{{\bf\eta}\in \Gamma_m:\ I({\bf\eta})=i\}$ and $\Gamma_*(i)=\bigcup_{m\geq0}\Gamma_m(i)$. Write $\Gamma_*=\bigcup_{i=1}^{P}\Gamma_*(i)$ for all finite admissible words.

\medskip

We now apply the aforementioned definition to a specific context, namely, open subsets in p.c.f. self-similar sets.
Let $(K,\{F_i\}_{i=1}^N,V_0)$ be a p.c.f. self-similar set. For $P\geq1$, let $\{\Omega_1,\Omega_2,\ldots,\Omega_P\}$ be a collection of non-empty open subsets in $K$ such that each $\Omega_i$ has a non-empty boundary with respect to the metric $d$, denoted as $D_i$. We refer to $D_i$ as the {\it geometric boundary} of $\Omega_i$. We assume that the collection $\{(\Omega_i,D_i)\}_{1\leq i\leq P}$ satisfies the following boundary graph-directed condition:

\medskip

\noindent {\bf BGD}:\quad \textit{for $1\leq i\leq P$ and $1\leq k\leq N$, if $\Omega_i\cap F_k(K)\neq\emptyset$ and $D_i\cap F_k(K)\neq\emptyset$, then there exists $1\leq j\leq P$ such that
\begin{equation*}
\Omega_i\cap F_k(K)=F_k(\Omega_j),\quad D_i\cap F_k(K)=F_k(D_j).
\end{equation*}}
\noindent {\bf Remark 1.} {\it Because $F_k(K)$ is arcwise-connected (see \cite[Theorem 1.6.2]{K}), the condition $\Omega_i\cap F_k(K)\neq\emptyset$ implies that either $F_k(K)\subset \Omega_i$ or $D_i\cap F_k(K)\neq\emptyset$. The BGD condition then guarantees that in the latter case there exists an index $j$ such that $\Omega_i\cap F_k(K)=F_k(\Omega_j)$.}

\medskip

Based on the configuration of $\{\Omega_{i}\}_{i=1}^{P}$, we define the directed graph on $\mathcal{
A}=\{1,\ldots,P\}$ as follows. For each pair $(i,j)$ in the BGD condition, we set a directed edge $\eta$ from $i$ to $j$ associated with the contraction map $\Phi_\eta:=F_k$. Let $\Gamma$ be the set of all such directed edges $\eta$ between vertices in $\mathcal A$.
 Consequently, we obtain a directed graph $(\mathcal A,\Gamma)$ and a set of contractions $\{\Phi_\eta\}_{\eta\in\Gamma}$.
Furthermore, the collection $\{D_i\}_{i=1}^{P}$ satisfies the equations \eqref{eqGD}, and thus, $\{D_i\}_{i=1}^{P}$ constitutes a collection of graph-directed self-similar sets.

\medskip

\noindent {\bf Remark 2.} {\it The BGD condition can in fact be relaxed: replacing ``$\Omega_i\cap F_k(K)=F_k(\Omega_j), D_i\cap F_k(K)=F_k(D_j)$" with only ``$\Omega_i\cap F_k(K)=F_k(\Omega_j)$". For clarity, we call this weaker version the $\widetilde{BGD}$ condition. }

\medskip

\begin{proposition} \label{thm 3,1BGD}
If $\{\Omega_1,\ldots,\Omega_P\}$ satisfies $\widetilde{BGD}$, then \eqref{eqGD} still holds.
\end{proposition}
\begin{proof}
``$D_i\supset\bigcup_{\eta\in\Gamma(i)}\Phi_\eta(D_{T(\eta)})$". Take $\eta\in \Gamma(i)$, write $T(\eta)=j$ and $\Phi_\eta=F_k$. For any $x\in F_k(D_j)$, we have $x\notin F_k(\Omega_j)=\Omega_i\cap F_k(K)$, hence $x\notin\Omega_i$. Moreover, every neighborhood of $x$ meets $F_k(\Omega_j)\subset\Omega_i$. These imply $x\in D_i$, so the inclusion ``$\supset$" holds.

``$D_i\subset\bigcup_{\eta\in\Gamma(i)}\Phi_\eta(D_{T(\eta)})$". Let $x\in D_i$. Then $x\notin \Omega_i$, and every neighborhood of $x$ meets $\Omega_i$. We claim there exists $k\in\{1,\ldots,N\}$ such that every neighborhood of $x$ meets $\Omega_i\cap F_k(K)$. Indeed, if $x\notin V_1$, then $x\in F_k(K\setminus V_0)$ for a unique $k$, and this $k$ satisfies the claim; if $x\in V_1$, then $x$ lies in finitely many $1$-cells of $K$, and we can choose one of them, say $F_k(K)$, such that every neighborhood of $x$ meets
$\Omega_i\cap F_k(K)$. By the claim, we have $\Omega_i\cap F_k(K)\neq\emptyset$ and $D_i\cap F_k(K)\neq\emptyset$. By $\rm \widetilde{BGD}$, there exists $j\in\{1,\ldots,P\}$ with $\Omega_i\cap F_k(K)=F_k(\Omega_j)$. Hence $x\in F_k(D_j)$, and the inclusion ``$\subset$" holds.
\end{proof}

Note that under $\rm \widetilde{BGD}$, Proposition \ref{thm 3,1BGD} gives $D_i\cap F_k(K)=F_k(D_j\cup V)$ for some $V\subset V_0$, whenever $\Omega_i\cap F_k(K)\neq\emptyset$ and $D_i\cap F_k(K)\neq\emptyset$. This is weaker than the identity $D_i\cap F_k(K)=F_k(D_j)$, required by BGD. Nevertheless, since \eqref{eqGD} remains valid, all subsequent arguments apply to $\rm\widetilde{BGD}$ as well. See Example \ref{subsec6.1}-5 for an example satisfying $\rm\widetilde{BGD}$ but not BGD.

In what follows, for ${\bf\eta}\in\Gamma_*$,
we write $\Omega_{\bf\eta}:=\Phi_{\bf\eta}(\Omega_{T({\bf\eta})})$ and $D_{\bf\eta}:=\Phi_{\bf\eta}(D_{T({\bf\eta})})$ for short.

\section { Weyl-Berry asymptotics: the irreducible case}\label{sec 4}
\setcounter{equation}{0}\setcounter{theorem}{0}

In this section, we will consider the Weyl-Berry spectral asymptotic for BGD domains $\{\Omega_i\}_{i=1}^P$ in a p.c.f. self-similar set $K$, equipped with a strongly recurrent self-similar Dirichlet form $(\mathcal E,\mathcal F)$, under the assumption that $\gamma_i=\gamma$ for all $1\leq i\leq N$. We will only look at the irreducible case and postpone the general case to the next section.

Let $(\mathcal A,\Gamma)$ be the associated directed graph. We write the {\it incidence matrix} of $(\mathcal A,\Gamma)$ as $A=(a_{ij})_{P\times P}$, which is a $P\times P$ non-negative matrix with $a_{ij}=\#\Gamma(i,j)$.  In this section, we assume that $A$ is {\it irreducible}, i.e. for any $i,j\in \{1,\ldots,P\}$, there exists $n\geq1$ such that $(A^n)_{ij}>0$.
We write $\Psi(A)$ the spectral radius of $A$ and $\tilde A=\frac1{\Psi(A)}A$ for normalization.

\medskip

\noindent{\bf Remark.}  {\it $1\leq\Psi(A)<N$.}

Since for large $n$, each $\Omega_i$ must contain at least one $n$-cell, the summation of each row of $A^n$ is strictly less than $N^n$, giving that $\Psi(A^n)<N^n$ and so $\Psi(A)<N$. $\Psi(A)\geq1$ is clear.

\medskip


Let $\mu$ be a self-similar measure on $K$ with probability weights $\mu_1,\ldots,\mu_N$.
Then the measure of the open sets $\Omega_i$ satisfy the following recursive formula:
\begin{equation}\label{genmeasureformula}
\mu(\Omega_i)=\sum\limits_{{\eta}\in\Gamma(i)}\mu_{T({\eta})}\mu(\Omega_{T({\eta})})+\sum\limits_{k\in\Sigma :\ K_k\subset \Omega_i}\mu_k\quad \text{for  }1\leq i\leq P.
\end{equation}

Let $\Omega$ be an open set in $K$. For a function $u\in C(\Omega)$, by considering $\Omega$ as a countable union of cells whose pairwise intersection is a set of finite points, we define the energy of $u$ on $\Omega$ to be the summation of energies of
$u$ on each of the cells, and denote it as $\mathcal E_{\Omega}[u]$ (might equal to $+\infty$). By virtue of \eqref{essn}, we
see that $\mathcal E_{\Omega}[u]$ does not depend on the partition of $\Omega$. Define
\begin{equation*}\mathcal F_{\Omega} = \{u \in
C(\Omega)\cap L^2(\Omega,\mu|_{\Omega}): \mathcal E_{\Omega}[u] < \infty\},
\end{equation*}
 where $\mu|_{\Omega}$ is the restriction of $\mu$ on $\Omega$.
 By polarization, we define
 \begin{equation*}
 \mathcal E_{\Omega}(u,v)=
\frac14\big(\mathcal E_{\Omega}[u+v]-\mathcal E_{\Omega}[u-v]\big) \quad\text{for
$u, v \in \mathcal F_{\Omega}$}.
\end{equation*}
It is direct to check that $(\mathcal E_{\Omega},\mathcal F_{\Omega})$ is a Dirichlet form on $L^2(\Omega,\mu|_{\Omega})$.
We also define $\mathcal F_{\Omega,0}$ to be the closure of $C_0(\Omega)\cap\mathcal F_{\Omega}$ under $\mathcal E^{1/2}_{\Omega,1}$-norm, where $C_0(\Omega)$ means the space of continuous functions compactly supported in $\Omega$ and $\mathcal E_{\Omega,1}[u]=\mathcal E_{\Omega}[u]+\int_{\Omega}u^2d\mu$ for $u\in\mathcal F_{\Omega}$. Let $\mathcal E_{\Omega,0}$ be the restriction of $\mathcal E_{\Omega}$ on $\mathcal F_{\Omega,0}\times \mathcal F_{\Omega,0}$, then $(\mathcal E_{\Omega,0},\mathcal F_{\Omega,0})$ also turns out to be a Dirichlet form on $L^2(\Omega,\mu|_{\Omega})$.

Let $\Delta_{i,D}$ (or $\Delta_{i,N}$) denote the Laplace operator of the form $(\mathcal E_{\Omega_i,0},\mathcal F_{\Omega_i,0})$ (or $(\mathcal E_{\Omega_i},\mathcal F_{\Omega_i})$) on $L^2(\Omega_i,\mu|_{\Omega_i})$ with Dirichlet (or Neumann) boundary conditions. For the Dirichlet case, since $\mathcal F_{\Omega_i,0}\subset \mathcal F$, the operator $-\Delta_{i,D}$ has compact resolvent. For the Neumann case, assuming in addition that $\Omega_i$ is connected, Proposition 4.3 in \cite{GQ} implies that ${\Omega_i}$ is bounded in the effective resistance metric, hence $-\Delta_{i,D}$ also possesses compact resolvent. Consequently, both $-\Delta_{i,D}$ and $-\Delta_{i,N}$ (under the connectivity assumption on $\Omega_i$) have purely discrete spectra contained in $[0,+\infty)$ with the only accumulation point at $+\infty$. In what follows, we shall always assume the domains $\{\Omega_i\}_{i=1}^P$ are connected whenever Neumann eigenvalue problems are discussed.

For $1\leq i\leq P$ and $*=D$ or $N$, we define the {\it eigenvalue counting function} of $-\Delta_{i,*}$ as
\begin{equation*}
\rho^{\Omega_i}_{*}(x)=\#\{k:\ k\leq x \text{ and $k$ is a positive eigenvalue of } -\Delta_{i,*}\},
\end{equation*}
where each eigenvalue is counted according to its multiplicity.

As in \eqref{eqecfK}, we also denote $\rho_{*}(x)$ the corresponding eigenvalue counting functions of $-\Delta_*$ on $K\setminus V_0$. By Theorem \ref{KigamiJFA}, we know that under the assumption that $\gamma_i=\gamma$ for all $1\leq i\leq N$, we have for $*=D \text{ or }N$,
\begin{equation}\label{Kigamiformula}
\rho_{*}(x)=G\Big(\frac{\log x}2\Big)x^{\frac{d_S}2}+O(1) \quad \text{as $x\rightarrow+\infty$},
\end{equation}
where $G$ is a positive, bounded, right-continuous, periodic function with period $T=-\log\gamma$, and  the {\it spectral exponent} $d_S=\frac{\log N}{-\log\gamma}$. 


Since $\tilde A$ is an irreducible non-negative matrix, its spectral radius $1$ is a single eigenvalue and the corresponding eigenvectors are strictly positive. Fix ${\bf b}=(b_1,\cdots,b_P)$ to be a right $1$-eigenvector of $\tilde A$. For any $1\leq i\leq P$, and any ${\bf\xi}\in \Gamma_m(i)$, $m\geq1$, define a set function $\kappa_i$ on $\{D_{\bf\xi}:\ {\bf\xi}\in \Gamma_*(i)\}$ by
\begin{equation*}
\kappa_i(D_{\bf\xi})=\frac1{\Psi(A)^{m}} b_{T({\bf\xi})}.
\end{equation*}
In a standard way, since
\begin{align*}
\sum_{\eta\in\Gamma(T({\bf\xi}))}\kappa_i(D_{{\bf\xi}\eta})&=\sum_{j=1}^P\sum_{\eta\in\Gamma(T({\bf\xi}),j)}\frac{1}{\Psi(A)^{m+1}}b_j=\frac{1}{\Psi(A)^{m+1}}\sum_{j=1}^Pa_{T({\bf\xi}),j}b_j\\
&=\frac{1}{\Psi(A)^{m}}(\tilde A{\bf b})_{T({\bf\xi})}=\frac{1}{\Psi(A)^{m}}b_{T({\bf\xi})}=\kappa_i(D_{\bf\xi}),
\end{align*}
 $\kappa_i$ extends to be a Borel measure on $D_i$ by the Kolmogorov extension theorem.
Note that $\kappa_i(D_i)=b_i$ for any $1\leq i\leq P$.

In the following, we write $\nu$ the $(\frac1N,\ldots,\frac1N)$-self-similar measure on $K$. Note that by \eqref{genmeasureformula}, it satisfies
\begin{equation}\label{genmeasureformula1}
\nu(\Omega_i)=\frac{1}{N}\Big(\sum\limits_{{\eta}\in\Gamma(i)}\nu(\Omega_{T({\eta})})+\#\{k\in\Sigma :\ K_k\subset \Omega_i\}\Big)\quad \text{for  }1\leq i\leq P.
\end{equation}
We refer to $\nu|_{\Omega_i}$ and $\kappa_i$ as {\it the spectral reference measures} on $\Omega_i$ and $D_i$ for $1\leq i\leq P$, respectively.

For the irreducible non-negative $P\times P$ matrix $A$, define for $i,j\in\{1,\ldots,P\}$, $t_{ij}=\min\ \{k\geq1:\ A^k(i,j)>0\}$. Let $\mathcal G_{i}$ be the subgroup of $\mathbb Z$ generated by $\{k\geq1:\ A^k(i,i)>0\}$, and $t_i\geq1$ be the generator of $\mathcal G_{i}$. Let $\varrho$ be the greatest common divisor of $t_1,\ldots,t_P$. Note that $\varrho$ is the generator of the subgroup in $\mathbb Z$ generated by $t_1,\ldots,t_P$.

The following is the main result in this section.
\begin{theorem}\label{thm-irreducible}
Assume $A$ is irreducible. Let $G$ be the same function as in \eqref{Kigamiformula}.

When $\Psi(A)>1$, there exist two bounded $\varrho T$-periodic functions $G_{*}$ for $*=D$ or $N$ such that for $i\in\{1,\ldots,P\}$, as $x\rightarrow+\infty$,
\begin{equation}\label{eqsecorder}
\rho^{\Omega_i}_{*}(x)=\nu(\Omega_i)G\Big(\frac{\log x}{2}\Big)x^{\frac{d_S}2}+\kappa_i(D_i)G_{*}\Big(\frac{\log x}{2}-t_{i1}T\Big) x^{\frac{d}2}+o\big(x^{\frac{d}2}\big),
\end{equation}
where $d_S=\frac{\log N}{-\log\gamma}$ and $d=\frac{\log\Psi(A)}{-\log\gamma}\in(0,d_S)$.

When $\Psi(A)=1$ (equivalently, each $D_i$ is a singleton in $K$), it holds that for $i\in\mathcal A$,
\begin{equation*}
\rho^{\Omega_i}_{*}(x)=G\Big(\frac{\log x}2\Big)x^{\frac{d_S}2}+O(1)\quad \text{as $x\rightarrow+\infty$}.
\end{equation*}
\end{theorem}

Before proceeding, we introduce three more types of auxiliary Dirichlet forms:

\medskip

1. $(\mathcal E_{\Omega_i},\mathcal F_{\Omega_i,0,0})$. Define $\mathcal F_{\Omega_i,0,0}=\{u\in \mathcal F_{\Omega_i,0}:\ u|_{\Omega_i\cap V_0}=0\}$ and restrict $\mathcal E_{\Omega_i}$ on $\mathcal F_{\Omega_i,0,0}\times\mathcal F_{\Omega_i,0,0}$.

\medskip

2. $(\mathcal E_{\Omega_i},\mathcal F'_{\Omega_i})$. Define $\mathcal F'_{\Omega_i}=\{u\in\mathcal F_{\Omega_i,0}:\ u|_{\Omega_i\cap V_1}=0\}$ and restrict $\mathcal E_{\Omega_i}$ on ${\mathcal F'_{\Omega_i}\times\mathcal F'_{\Omega_i}}$.

\medskip

 3.  $(\widetilde {\mathcal E}_{\Omega_i},\widetilde {\mathcal F}_{\Omega_i})$. Define
\begin{align*}
\widetilde {\mathcal F}_{\Omega_i}=\Big\{u: \Omega_i\setminus\ V_1\rightarrow \mathbb R:\ &\text{$u\circ F_k\in \mathcal F_{\Omega_j}$ for $k\in\Sigma$ such that $\Omega_i\cap K_k =\Phi_{\eta}(\Omega_j)$ for some}\\
& \text{$j\in\mathcal A$ and $\eta\in\Gamma(i)$; $u\circ F_k\in \mathcal F$  for other $k\in\Sigma$}\Big\},
\end{align*}
and let $\widetilde {\mathcal E}_{\Omega_i}$ be the form on $\widetilde {\mathcal F}_{\Omega_i}$ defined as
\begin{equation*}
\widetilde {\mathcal E}_{\Omega_i}(u,v)=\sum_{\eta:\eta\in\Gamma(i)}\frac1{r_\eta}\mathcal E_{\Omega_{T(\eta)}}(u\circ \Phi_\eta,v\circ \Phi_\eta)+\sum_{k: K_k \subset \Omega_i}\frac1{r_k}\mathcal E(u\circ F_k,v\circ F_k),
\end{equation*}
where $r_{\eta}:=r_k$ for $k\in\{1,\ldots,N\}$ such that $\Phi_{\eta}=F_k$. Note that by regarding $\mathcal F_{\Omega_i}$ as a subspace of $L^2(\Omega_i\setminus V_1,\mu)$, we have $\mathcal F_{\Omega_i}\subset \widetilde {\mathcal F}_{\Omega_i}$ and ${\mathcal E}_{\Omega_i}=\widetilde {\mathcal E}_{\Omega_i}|_{\mathcal F_{\Omega_i}\times\mathcal F_{\Omega_i}}$.

 Denote by $\rho(x;\mathcal E_{\Omega_i},\mathcal F_{\Omega_i,0,0})$, $\rho(x; \mathcal E_{\Omega_i},\mathcal F'_{\Omega_i})$ and $\rho(x; \widetilde {\mathcal E}_{\Omega_i},\widetilde {\mathcal F}_{\Omega_i})$ the corresponding eigenvalue counting functions  associated with the above Dirichlet forms. Then by the Dirichlet-Neumann bracketing method (see \cite[Proposition 2.7]{Me} and also \cite[Corollary 4.7]{KL}), we have for any $x$,
\begin{align}\label{DBCineq}
\begin{cases}
&\rho(x;\mathcal E_{\Omega_i},\mathcal F_{\Omega_i,0,0})\leq \rho(x;\mathcal E_{\Omega_i},\mathcal F_{\Omega_i,0})\leq \rho(x;\mathcal E_{\Omega_i},\mathcal F_{\Omega_i,0,0})+\dim \mathcal F_{\Omega_i,0}/\mathcal F_{\Omega_i,0,0},\\
&\rho(x;\mathcal E_{\Omega_i},\mathcal F'_{\Omega_i})\leq \rho(x;\mathcal E_{\Omega_i},\mathcal F_{\Omega_i,0})\leq \rho(x;\mathcal E_{\Omega_i},\mathcal F'_{\Omega_i})+\dim \mathcal F_{\Omega_i,0}/\mathcal F'_{\Omega_i},\\
&\rho(x;\mathcal E_{\Omega_i},\mathcal F_{\Omega_i})\leq \rho(x; \widetilde {\mathcal E}_{\Omega_i},\widetilde {\mathcal F}_{\Omega_i})\leq \rho(x;\mathcal E_{\Omega_i},\mathcal F_{\Omega_i})+\dim \widetilde {\mathcal F}_{\Omega_i}/\mathcal F_{\Omega_i},
\end{cases}
\end{align}
where $\dim \mathcal F_{\Omega_i,0}/\mathcal F_{\Omega_i,0,0}$ is the dimension of the space of functions in $\mathcal F_{\Omega_i,0}$ with prescribed values on $\Omega_i\cap V_0$ and harmonic elsewhere, thus is equal to
$\#(\Omega_i\cap V_0)\leq \#V_0$, and similarly, $\dim \mathcal F_{\Omega_i,0}/\mathcal F'_{\Omega_i}\leq N\cdot \#V_0$, $\dim \widetilde {\mathcal F}_{\Omega_i}/\mathcal F_{\Omega_i}\leq N\cdot \#V_0$.

\begin{lemma}\label{lemmaKL}For $1\leq i\leq P$, we have for any $x$,
\begin{align*}
\rho(x;\mathcal E_{\Omega_i},\mathcal F'_{\Omega_i})&=\sum_{{\eta}: {\eta}\in\Gamma(i)}\rho(\gamma^2x;\mathcal E_{\Omega_{T({\eta})}},\mathcal F_{\Omega_{T({\eta})},0,0})+\rho_D(\gamma^2x)\cdot \#\{k:\ K_k\subset \Omega_i\},\\
\rho(x;\widetilde {\mathcal E}_{\Omega_i},\widetilde {\mathcal F}_{\Omega_i})&=\sum_{{\eta}: {\eta}\in\Gamma(i)}\rho(\gamma^2x;{\mathcal E}_{\Omega_{T({\eta})}},{\mathcal F}_{\Omega_{T({\eta})}})+\rho_N(\gamma^2x)\cdot \#\{k:\ K_k\subset \Omega_i\}.
\end{align*}
\end{lemma}
\begin{proof}
The idea of the proof is from \cite[Proposition 6.2]{KL}. Let $f$ be an eigenfunction of the Dirichlet form $(\mathcal E_{\Omega_i},\mathcal F'_{\Omega_i})$ with eigenvalue $\lambda$. By the BGD condition, $\Omega_i$ is a union of several non-overlapping parts, i.e. $\{\Phi_{\eta}(\Omega_{T({\eta})}):\ {\eta}\in\Gamma(i)\}$ and $\{K_k:\ 1\leq k\leq N, K_k\subset \Omega_i\}$. For any $g\in \mathcal F'_{\Omega_i}$, we denote $g_{\eta}=g\circ \Phi_{\eta}$ for $\eta\in \Gamma(i)$ and $g_k=g\circ F_{k}$ for $1\leq k\leq N$ such that $K_k\subset \Omega_i$. Note that $g_{\eta}\in\mathcal F_{\Omega_{T(\eta),0,0}}$, and $g_k\in\mathcal F_0$. So we have
\begin{equation}\label{equation1forE}
\mathcal E_{\Omega_i}(f,g)=\sum_{{\eta}: {\eta}\in\Gamma(i)}\frac{1}{r_{\eta}}\mathcal E_{\Omega_{T({\eta})}}(f_{\eta},g_{\eta})+\sum_{k: K_k\subset\Omega_i}\frac{1}{r_k}\mathcal E(f_k,g_k),
\end{equation}
and
\begin{equation}\label{equation2forL2}
\int_{\Omega_i}fgd\mu=\sum_{{\eta}: {\eta}\in\Gamma(i)}\mu_{\eta}\int_{\Omega_{T(\eta)}}f_{\eta}g_{\eta}d\mu+\sum_{k: K_k\subset\Omega_i}\mu_k\int_{K}f_kg_kd\mu,
\end{equation}
where $r_\eta=r_k$, $\mu_\eta=\mu_k$ for $k\in\{1,\ldots,N\}$ such that $\Phi_{\eta}=F_k$.

Hence by $\mathcal E_{\Omega_i}(f,g)=\lambda\int_{\Omega_i}fgd\mu$ and the arbitrariness of $g$, we see from \eqref{equation1forE} and \eqref{equation2forL2} that for any ${\eta}\in\Gamma(i)$, $f_{\eta}$ is an eigenfunction of the Dirichlet form $(\mathcal E_{\Omega_{T({\eta})}},\mathcal F_{\Omega_{T({\eta}),0,0}})$ with eigenvalue $\gamma^2\lambda$; for any $k$ with $K_k\subset\Omega_i$, $f_k$ is an eigenfunction of the Dirichlet form $(\mathcal E,\mathcal F_0)$ with eigenvalue $\gamma^2\lambda$. Together with a converse consideration, we have for any $x$,
\begin{equation*}
\rho(x;\mathcal E_{\Omega_i},\mathcal F'_{\Omega_i})=\sum_{{\eta}: {\eta}\in\Gamma(i)}\rho(\gamma^2x;\mathcal E_{\Omega_{T({\eta})}},\mathcal F_{\Omega_{T({\eta})},0,0})+\sum_{k: K_k\subset \Omega_i}\rho_D(\gamma^2x).
\end{equation*}
This proves  the first line of equalities. The second follows in a similar way.
\end{proof}

Combining Lemma \ref{lemmaKL} and \eqref{DBCineq}, we immediately have the following corollary.
\begin{corollary}\label{thmcor}
For $1\leq i\leq P$, by letting $M=N\cdot \#V_0$, we have for any $x$,
\begin{equation*}
\rho^{\Omega_i}_{D}(x)-M\leq \sum_{{\eta}: {\eta}\in\Gamma(i)}\rho^{\Omega_{T(\eta)}}_{D}(\gamma^2x)+\rho_D(\gamma^2x)\cdot\#\{k: K_k\subset \Omega_i\}\leq \rho^{\Omega_i}_{D}(x)+M
\end{equation*}
for the Dirichlet case, and
\begin{equation*}
\rho^{\Omega_i}_{N}(x)\leq\sum_{{\eta}: {\eta}\in\Gamma(i)}\rho^{\Omega_{T(\eta)}}_{N}(\gamma^2x)+\rho_N(\gamma^2x)\cdot\#\{k: K_k\subset \Omega_i\}\leq \rho^{\Omega_i}_{N}(x)+M
\end{equation*}
for the Neumann case.
\end{corollary}
To simplify notations, for $1\leq i\leq P$ and $*=D$ or $N$, we denote $s_i=\#\{k: K_k\subset\Omega_i\}$, $c_i=\nu(\Omega_i)$, and write
\begin{equation*}
\begin{cases}
{\bf s}=(s_1,\ldots,s_P)^T,\\
{\bf c}=(c_1,\ldots,c_P)^T,\\
{\bf 1}=(1,\ldots,1)^T,\\
{\bf 0}=(0,\ldots,0)^T,
\end{cases}
\end{equation*}
and
\begin{equation*}
{\bf\rho}^{\Omega}_*(x)=(\rho^{\Omega_1}_*(x),\ldots,\rho^{\Omega_P}_*(x))^T.
\end{equation*}
 It follows from \eqref{genmeasureformula1} that
\begin{equation}\label{equationb}
{\bf c}=\frac1N\left(A{\bf c}+{\bf s}\right).
\end{equation}
For $*=D$ or $N$, we define \begin{equation*}
{\bf\varphi}(x)=\begin{cases}{\bf\rho}^{\Omega}_*(x)-G\Big(\frac{\log x}{2}\Big)x^{\frac{d_S}2}{\bf c} \quad &x\geq e,\\
{\bf0} \quad &0\leq x<e.
\end{cases}
\end{equation*}
\begin{lemma}\label{lemma4.3}We have
\begin{equation}\label{eqvarphi}
{\bf\varphi}(x)=A{\bf\varphi}(\gamma^2x)+O({\bf 1})\quad \text{as }x\rightarrow+\infty,
\end{equation}
where $O({\bf 1})$ stands for $O(1){\bf 1}$.
\end{lemma}
\begin{proof}
By Corollary \ref{thmcor}, we have
\begin{equation*}
{\bf\rho}^{\Omega}_*(x)=A{\bf\rho}^{\Omega}_*(\gamma^2x)+\rho_*(\gamma^2x){\bf s}+O({\bf 1})\quad \text{as }x\rightarrow+\infty.
\end{equation*}
Combining this with \eqref{Kigamiformula} and \eqref{equationb}, we obtain that as $x\rightarrow+\infty$,
\begin{align*}
{\bf\varphi}(x)&={\bf\rho}^{\Omega}_*(x)-G\Big(\frac{\log x}{2}\Big)x^{\frac{d_S}2}{\bf c}\\
&=A{\bf\rho}^{\Omega}_*(\gamma^2x)+\rho_*(\gamma^2x){\bf s}-\frac1NG\Big(\frac{\log x}{2}\Big)x^{\frac{d_S}2}(A{\bf c}+{\bf s})+O({\bf 1})\\
&=A\Big({\bf\rho}^{\Omega}_*(\gamma^2x)-G\Big(\frac{\log (\gamma^2x)}{2}\Big)(\gamma^2x)^{\frac{d_S}2}{\bf c}\Big)+O({\bf 1})\\
&=A{\bf\varphi}(\gamma^2x)+O({\bf 1}),
\end{align*}
where in the third equality we use the facts that $\gamma^{d_S}=\frac1N$ and $G$ is $T$-periodic ($T=-\log\gamma$).
\end{proof}

\noindent {\bf Proof of Theorem \ref{thm-irreducible}.}
{\it Case 1. $\Psi(A)=1$.} We claim that in this case, $A$ is a permutation of the identity matrix. Since $A$ is irreducible, $A$ can not have zero columns. Recall that $\bf b$ is a right $1$-eigenvector of $A$. By summing up all the entries in both sides of ${\bf b}=A{\bf b}$, we see that each row or column of $A$ is a unit vector with one entry $1$ and others zero.

This gives that each $\Omega_i$ is of the form $K\setminus\{p\}$ for a singleton $p\in K$. We note that in this case $\nu(\Omega_i)=1$ for each $i\in\mathcal A$.
By \eqref{Kigamiformula}, as $x\rightarrow+\infty$,
\begin{equation*}
{\rho}^{\Omega_i}_{*}(x)=G\Big(\frac{\log x}2\Big)x^{\frac{d_S}2}+O(1)\quad\text{for }i=1,\ldots,P.
\end{equation*}


\medskip

{\it Case 2. $\Psi(A)>1$.} By $\Psi(A)<N$, we have
\begin{equation*}
d=\frac{\log\Psi(A)}{-\log\gamma}\in(0,d_S).
\end{equation*}
For $*=D$ or $N$, let us introduce two vectors of functions on $\mathbb R$:
\begin{equation}\label{equfphi}
\begin{cases}
{\bf f}(t)=e^{-dt}{\bf\varphi}(e^{2t}),\\
{\bf z}(t)=e^{-dt}\big({\bf\varphi}(e^{2t})-A{\bf\varphi}(\gamma^2e^{2t})\big).
\end{cases}
\end{equation}
We can check that
\begin{equation*}
{\bf f}(t)=\tilde A{\bf f}(t-T)+{\bf z}(t),
\end{equation*}
where $T=-\log\gamma$.
By Lemma \ref{lemma4.3}, we see that ${\bf z}(t)=e^{-dt}O({\bf1})$ as $t\rightarrow+\infty$ and ${\bf z}(t)={\bf 0}$ for $t<\frac12$.

By a corollary of a vector-valued renewal theorem (see Corollary \ref{thm7.3} in the Appendix), we have
\begin{equation*}
\lim_{t\rightarrow+\infty}\left(\left(
                                                                           \begin{array}{c}
                                                                             f_1(t+t_{11}T)\\
                                                                            \ldots \\
                                                                            f_P(t+t_{P1}T)\\
                                                                           \end{array}
                                                                         \right)
-\varrho B\sum_{k\in\mathbb Z}\left(\begin{array}{c}
                                                                             z_1(t+t_{11}T+k\varrho T)\\
                                                                            \ldots \\
                                                                            z_P(t+t_{P1}T+k\varrho T)\\
                                                                           \end{array}
                                                                         \right)\right)=0
\end{equation*}
with the matrix $B=\frac1T{\bf b}{{\bf d}^T}$, where $\bf d$ is the unique positive left $1$-eigenvector of $\tilde A$ such that ${{\bf d}^T}{\bf b}=1$.

Define $G_*(t):=\frac{\varrho}T{{\bf d}^T}\sum_{k\in\mathbb Z}\left(\begin{array}{c}
                                                                             z_1(t+t_{11}T+k\varrho T)\\
                                                                            \ldots \\
                                                                            z_P(t+t_{P1}T+k\varrho T)\\
                                                                           \end{array}
                                                                         \right)$ on $\mathbb R$. Then ${G}_*$ is a bounded $\varrho T$-periodic function satisfying
\begin{equation}\label{eqeq1}
\lim_{t\rightarrow+\infty}\big(f_i(t+t_{i1}T)-b_iG_*(t)\big)=0,\quad i=1,\ldots,P.
\end{equation}
Changing $f_i$ back to $\varphi_i$ and $t$ back to $x$ in \eqref{eqeq1} through \eqref{equfphi}, we obtain
\begin{equation*}
\varphi_i(x)=b_iG_*\Big(\frac{\log x}{2}-t_{i1}T\Big)x^{\frac d2}+o(x^{\frac d2}), \quad\text{ as }x\rightarrow+\infty.
\end{equation*}
Hence by the definition of $\varphi_i$, we have
\begin{equation*}
\rho^{\Omega_i}_*(x)=c_iG\Big(\frac{\log x}{2}\Big)x^{\frac{d_S}2}+b_iG_*\Big(\frac{\log x}{2}-t_{i1}T\Big)x^{\frac d2}+o(x^{\frac d2}),\quad\text{ as }x\rightarrow+\infty,
\end{equation*}
which proves \eqref{eqsecorder}.
\hspace{11cm} $\square$
%


\section { Weyl-Berry asymptotics: the general case}\label{sec 5}
\setcounter{equation}{0}\setcounter{theorem}{0}

We then turn to consider the general case that the incidence  matrix $A$ might be reducible.
For $i,j\in \mathcal A$ (allowing $i=j$), we say that $i$ {\it has access to} $j$, denoted as $i\rightarrow j$, if there is an admissible word ${\bf\eta}\in \Gamma_*$ such that $I({\bf\eta})=i$ and $T({\bf\eta})=j$. For two non-empty sets $I,J\subset \mathcal A$, write $I\rightarrow J$ if there exist $i\in I$ and $j\in J$ such that $i\rightarrow j$.

We say that $i$ and $j$ {\it communicate}, denoted as $i\leftrightarrow j$, if $i\rightarrow j$ and $j\rightarrow i$. We call a non-empty subset $J\subset\mathcal A$ a (communicating) {\it class} if for any $i,j\in J$ and $k\in \mathcal A\setminus J$, $i\leftrightarrow j$ but $i\nleftrightarrow k$.
In this way, $\mathcal A$ is separated into classes and singletons that do not belong to any class. Note that a class may also be a singleton. Also, since we assume that each $\Gamma(i)\neq \emptyset$, $\mathcal A$ has at least one class. Further, any class $J$ induces a strongly connected subgraph of $(\mathcal A,\Gamma)$ with vertex set $J$, associated with an incidence matrix $A_J$, a submatrix of $A$. For simplicity, we refer to the spectral radius of $A_J$ as the spectral radius of $J$.

If a class $J$ has a spectral radius equal to $\Psi(A)$, then we call $J$ a {\it basic class}. Basic classes can further be separated based on different heights. Precisely, we call a collection of basic classes $\{J_1,J_2,\ldots,J_n\}$ a {\it basic chain} if $J_k\rightarrow J_{k+1}$ for any $1\leq k\leq n-1$. We refer to $n$ as the {\it length} of this basic chain. A basic class $J$ is said to have {\it height} $m$ (for integer $m\geq0$) if $m+1$ is the maximal length of all basic chains beginning with $J$. For $m\geq0$, denote by $\mathcal S_m$ the collection of basic classes with height $m$. We define $\mathcal S=\cup_{m\geq0}\mathcal S_m$.

For each basic class $J$ and $i\in J$, let $\mathcal G_{J,i}$ be the subgroup of $\mathbb Z$ generated by $\{k\geq1:\ A_J^k(i,i)>0\}$, and let $t_i(J)\geq1$ be the generator of $\mathcal G_{J,i}$. Let $\varrho_J$ be the greatest common divisor of $\{t_i(J)\}_{i\in J}$. For $j\in\mathcal A$, if $j\rightarrow\mathcal S$,
define
\begin{equation}\label{defmj}
m_j=\max\{\text{ height of }J:\ i\rightarrow J,J\in\mathcal S\},
\end{equation}
and
\begin{equation}\label{defrhoj}
\varrho_j=\text{ the least common multiple of }\{\varrho_J:\ j\rightarrow J,J\in\mathcal S\}.
\end{equation}
Denote $d=\frac{\log\Psi(A)}{-\log\gamma}$ as before. Note that when $d=0$, all classes are basic classes and $j\rightarrow\mathcal S$ for all $j\in \mathcal A$.



\begin{theorem}\label{reducible}
Let $j\in\mathcal A$, and $G$ be the same function as in \eqref{Kigamiformula}.

(1). Assume $d>0$. If $j\rightarrow\mathcal S$, then there exist two $\varrho_j T$-periodic functions $G_{j,*}$ for $*=D$ or $N$ such that as $x\rightarrow+\infty$,
\begin{equation*}
\rho^{\Omega_j}_*(x)=\nu(\Omega_j)G\Big(\frac{\log x}{2}\Big)x^{\frac{d_S}2}+G_{j,*}\Big(\frac{\log x}{2}\Big)x^{\frac{d}2}(\log x)^{m_j} +o\Big(x^{\frac{d}2}(\log x)^{m_j}\Big).
\end{equation*}

(2). Assume $d>0$. If $j\nrightarrow\mathcal S$, then as $x\rightarrow+\infty$,
\begin{equation*}
\rho^{\Omega_j}_*(x)=\nu(\Omega_j)G\Big(\frac{\log x}{2}\Big)x^{\frac{d_S}2}+o(x^{\frac d2}).
\end{equation*}

(3). Assume $d=0$.  Then as $x\rightarrow+\infty$,
\begin{equation*}
\rho^{\Omega_j}_*(x)=\nu(\Omega_j)G\Big(\frac{\log x}{2}\Big)x^{\frac{d_S}2}+O\Big((\log x)^{m_j}\Big).
\end{equation*}
\end{theorem}

\noindent{\bf Remark.}
{\it In fact, in Case (2), we can still obtain an exact second-order term of $\rho^{\Omega_j}_*(x)$ by considering the classes that $j$ has access to. To be precise, it suffices to consider the subgraph induced by $(\mathcal A,\Gamma)$ on the subset $\{j\}\cup\{i\in \mathcal A: j\rightarrow i\}$ , which falls under Case (1) or Case (3). In Case (3), we are not able to obtain a periodic function for the second-order term.}

\begin{proof}
We first assume $d>0$. We define $\bf f$ and $\bf z$ to be the same as in \eqref{equfphi}. Then $\bf f$ satisfies the vector-valued
renewal equation ${\bf f}(t)=\tilde A{\bf f}(t-T)+{\bf z}(t)$ but with $\tilde A$ not necessarily irreducible.

If $j\rightarrow \mathcal S$, by applying Theorem \ref{thmHN}, there exists a $\varrho_j T$-periodic function $G_{j,*}$ such that as $t\rightarrow+\infty$,
\begin{equation*}
f_j(t)=(2t)^{m_j}G_{j,*}(t)+o(t^{m_j}).
\end{equation*}
Equivalently, as $x\rightarrow+\infty$,
\begin{equation*}
\rho^{\Omega_j}_*(x)=\nu(\Omega_j)G\Big(\frac{\log x}2\Big)x^{\frac{d_S}2}+G_{j,*}\Big(\frac{\log x}2\Big)x^{\frac d2}(\log x)^{m_j}+o\Big( x^{\frac{d}2}(\log x)^{m_j}\Big),
\end{equation*}
which proves Case (1).

If $j\nrightarrow \mathcal S$, then, still by applying Theorem \ref{thmHN}, we obtain
\begin{equation*}
\lim_{t\rightarrow+\infty}f_j(t)=0,
\end{equation*}
or equivalently, as $x\rightarrow+\infty$,
\begin{equation*}
\rho^{\Omega_j}_*(x)=\nu(\Omega_j)G\Big(\frac{\log x}2\Big)x^{\frac{d_S}2}+o(x^{\frac d2}),
\end{equation*}
which proves Case (2).

\medskip

We then prove Case (3), i.e. $d=0$. It suffices to prove that $f_j(t)=O(t^{m_j})$ as $t\rightarrow+\infty$.

In this case, all classes have spectral radius $1$, and hence all classes are basic classes. We prove the result by induction on $m_j$.

If $m_j=0$, then by Theorem \ref{thm-irreducible} (applying $d=0$), we have for $i\in \mathcal S_0$, $f_i(t)=O(1)$ as $t\rightarrow+\infty$. If $j\notin\mathcal S_0$, then we can write $f_j(t)$ as a finite linear combination of $f_i(t-kT)$ and $z_{i'}(t-k'T)$ with $i\in\mathcal S_0$, $i'\in\mathcal A$, and $k,k'\in\mathbb Z$. Consequently, $f_j(t)=O(1)$ as $t\rightarrow+\infty$.

Inductively, for $m\geq0$, assume for all $i$ with $m_i\leq m$, it holds that $f_i(t)=O(t^{m_i})$ as $t\rightarrow+\infty$. Consider a class $J\in\mathcal S_{m+1}$. Let $I=\{k\in \mathcal A\setminus J: J\rightarrow k\}$. Clearly, for each $k\in I$, $m_k\leq m$. Denote $U$ to be the sub-matrix of $A$ associated with the accesses from $J$ to $I$. Without loss of generality, assume $J=\{1,\ldots,s\}$. Denote ${\bf f}_1=(f_1(t),\ldots,f_s(t))^T$ and ${\bf z}_1=(z_1(t),\ldots,z_s(t))^T$. Also, denote ${\bf f}_2(t)$ as the vector of functions associated with $I$. We have
\begin{equation*}
{\bf f}_1(t)=A_{J}{\bf f}_1(t-T)+\tilde{\bf z}_1(t),
\end{equation*}
with $\tilde{\bf z}_1(t):=U{\bf f}_2(t-T)+{\bf z}_1(t)$. Iteratively, we have
\begin{equation*}
{\bf f}_1(t)=\sum_{k=0}^{[t/T]}A_{J}^k\tilde{\bf z}_1(t-kT),
\end{equation*}
where we use the fact that $\tilde{\bf z}_1(t)={\bf 0}$ for $t\leq0$.

Since each entry of $\tilde{\bf z}_1(t)$ is of order $O(t^m)$ as $t\rightarrow+\infty$, using the fact that $\frac1n\sum_{k=0}^{n}{A}^k_{J}\rightarrow M_{J}$, as $n\rightarrow +\infty$ for some matrix $M_{J}$, we have
\begin{equation*}
{\bf f}_1(t)=\Big[\frac tT\Big]\cdot\frac1{[\frac tT]}\sum_{k=0}^{[t/T]}A_{J}^k\tilde{\bf z}_1(t-kT)=O(t^{m+1}){\bf 1}, \qquad t\rightarrow+\infty,
\end{equation*}
which proves Case (3) for $j\in \mathcal S$.

If $j\notin \mathcal S$, then we can write $f_j(t)$ as a finite linear combination of $f_i(t-kT)$ and $z_{i'}(t-k'T)$ with $i\in\mathcal S$, $i'\in\mathcal A$, and $k,k'\in\mathbb Z$, which still implies $f_j(t)=O(t^{m_j})$ as $t\rightarrow+\infty$. This completes the proof of Case (3).

\end{proof}

\section{Examples}\label{sec 6}
 \setcounter{equation}{0}\setcounter{theorem}{0}
In this section, we present several examples to  illustrate Theorems  \ref{thm-irreducible} and  \ref{reducible}, as well as some further remarks.

\subsection{Example: Sierpi{\'n}ski gasket}\label{subsec6.1}
Let $p_1=(0,0)$, $p_2=(1,0)$, $p_3=(\frac12,\frac{\sqrt{3}}2)$ be the three vertices of an equilateral triangle in $\mathbb R^2$. Let $K$ be the {\it Sierpi{\'n}ski gasket} in $\mathbb R^2$, generated by the IFS $\{F_i\}_{i=1}^3$ defined by $F_i(x)=\frac 12 (x-p_i)+p_i$ for $i=1,2,3$, and let $V_0=\{p_1,p_2,p_3\}$. Let $\mu$ be the $\frac{\log3}{\log2}$-dimensional Hausdorff measure on $K$.
The standard Dirichlet form $(\mathcal{E},\mathcal{F})$ on $L^2(K,\mu)$  satisfies the self-similar identity \cite{K1}, with $r_i=\frac35$ for $i=1,2,3$,
\begin{equation*}
\mathcal E[u]=\frac53\sum_{i=1}^3\mathcal E[u\circ F_{i}], \quad \forall u\in \mathcal{F}.
\end{equation*}
Then $\gamma_i=\gamma=\frac1{\sqrt{5}}$, $T=\frac{\log 5}2$ and $d_S=\frac{2\log 3}{\log5}$.
In this subsection, we consistently use $G$ to represent the $\frac{\log 5}2$-periodic function in Theorem \ref{thm-KL}, which is bounded, positive (away from zero), right-continuous.

Arbitrarily pick two distinct points $p,q$ in $V_*=\bigcup_{|\omega|=0}^{\infty}F_{\omega}(V_0)$. Let $L$ denote the straight line passing through $p$ and $q$. The line $L$ separates the plane into two disjoint (open) parts, say $H_1$ and $H_2$.
As established in our previous work \cite[Proposition 7.2]{GQ}, both $H_1\cap K$ and $H_2\cap K$ (if non-empty) are BGD domains, see Figure \ref{figureSGcut}. So we can apply our results to compute the spectral asymptotics of the Laplacians on these domains.

For simplicity, we illustrate two particular situations.

\medskip

{\bf1.} $ p=p_1$, $ q=p_2$. Consider the open set $\Omega=K\setminus\overline{p_1p_2}$, see Figure \ref{figsg1}.
\begin{figure}[h]
	\includegraphics[width=6cm]{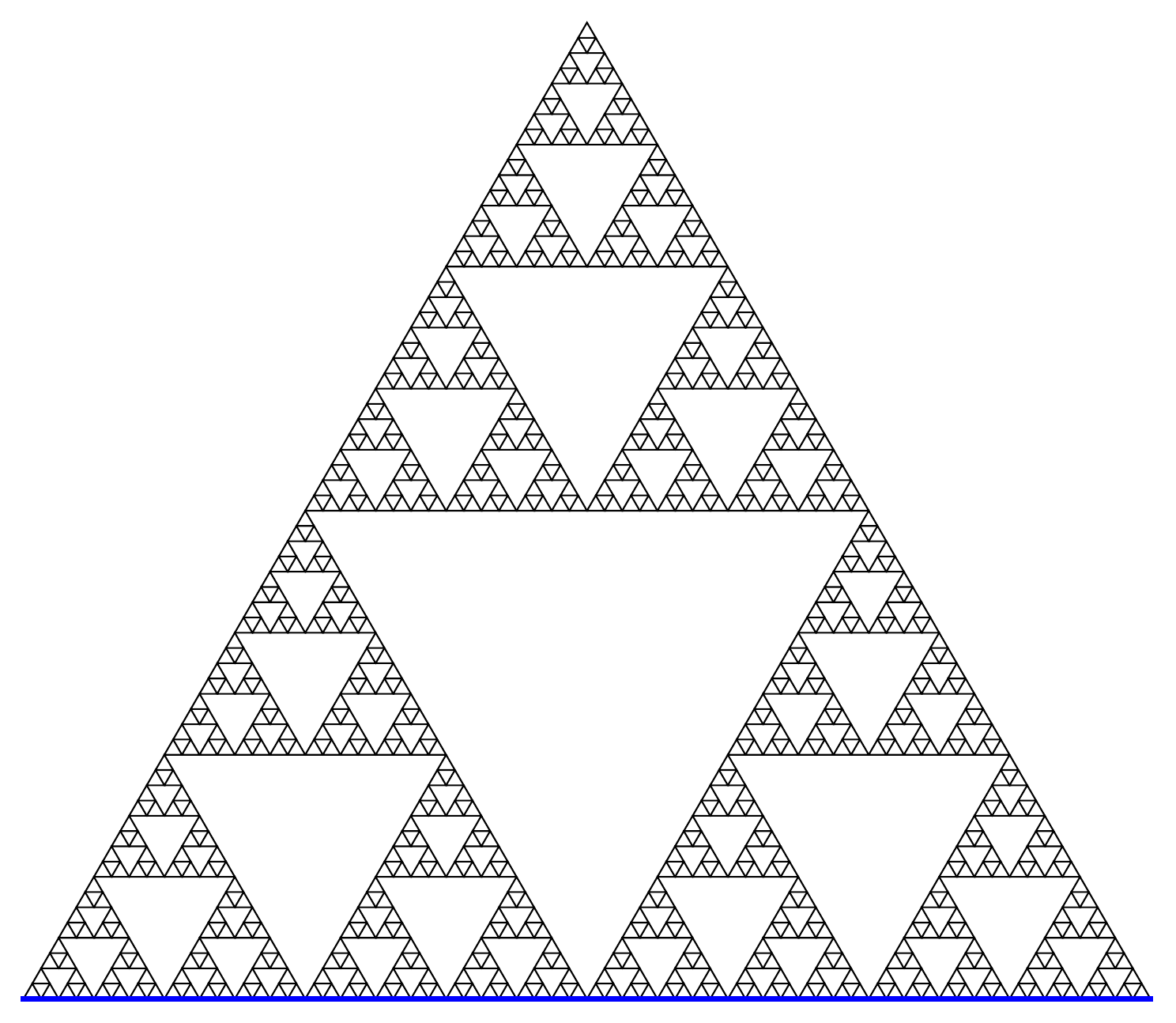}
	\caption{$\Omega$ in Example \ref{subsec6.1}-1} \label{figsg1}
\end{figure}

 Recall that the exact spectrum of the Laplacian on $\Omega$ with either Dirichlet or Neumann boundary conditions on $\overline{p_1p_2}$ (strictly speaking, on the resistance boundary $\widetilde\Omega\setminus\Omega$, where $\widetilde\Omega$ is the completion of $\Omega$ under the resistance metric, see \cite[Section 4]{GQ}) has been studied in detail by the second author in \cite{Q} using a spectral decimation method.

The Dirichlet eigenvalues are separated into three types: $\mathcal L$ for {\it localized eigenvalues} corresponding to eigenfunctions supported in $\Omega$; $\mathcal P$ for {\it primitive eigenvalues} corresponding to global supported symmetric (or skew-symmetric) eigenfunctions; $\mathcal M$ for {\it miniaturized eigenvalues} corresponding to local supported eigenfunctions generated by contracting skew-symmetric primitive eigenfunctions to the bottom of $\Omega$. Let $\rho^{\mathcal L}(x),\rho^{\mathcal P}(x)$ and $\rho^{\mathcal M}(x)$ denote the eigenvalue counting functions corresponding to the localized, primitive, and miniaturized eigenvalues, respectively.
Let $\rho_D(x)$ denote the eigenvalue counting function for the Laplacian on $K\setminus V_0$ with Dirichlet boundary conditions on $V_0$.
 It is proved in \cite{Q} that as $x\rightarrow+\infty$,
\begin{equation*}\rho_D(x)-\rho^{\mathcal L}(x)=O\Big(x^{\frac{\log 2}{\log 5}}\log x\Big),
\end{equation*}
$\rho^{\mathcal P}(x)=O\Big(x^{\frac{\log 2}{\log 5}}\Big)$, and $\rho^{\mathcal M}(x)=O\Big(x^{\frac{\log 2}{\log 5}}\log x\Big)$. Since $\rho_D^{\Omega}(x)=\rho^{\mathcal L}(x)+\rho^{\mathcal P}(x)+\rho^{\mathcal M}(x)$ and by \eqref{Kigamiformula}, $\rho_D(x)=G(\frac{\log x}2)x^{\frac{\log 3}{\log 5}}+O(1)$, one has as $x\rightarrow+\infty$,
\begin{equation}\label{eqOmegaD}
\rho_D^{\Omega}(x)=G\Big(\frac{\log x}2\Big)x^{\frac{\log 3}{\log 5}}+O\Big(x^{\frac{\log 2}{\log 5}}\log x\Big).
\end{equation}
For the Neumann boundary condition case, the spectral asymptotic of $\rho_N^{\Omega}(x)$ is similar to \eqref{eqOmegaD}.

Numerical experiments suggest that $\rho_D^{\Omega}(x)$ should have an explicit formula (see \cite[Conjecture 8.2]{Q}): there exists a bounded (away from zero) $\frac{\log 5}{2}$-periodic non-constant function $G_1$ such that as $x\rightarrow+\infty$,
\begin{equation*}
\rho_D^{\Omega}(x)=G\Big(\frac{\log x}2\Big)x^{\frac{\log 3}{\log 5}}+G_1\Big(\frac{\log x}2\Big)x^{\frac{\log 2}{\log 5}}+o\Big(x^{\frac{\log 2}{\log 5}}\Big).
\end{equation*}

By applying Theorem \ref{thm-irreducible}, we can nearly confirm the above conjecture affirmatively. We can establish the existence of a bounded $\frac{\log 5}{2}$-periodic function $G_1$, and confirm that $G_1\leq 0$ and is not identically zero. However, whether $G_1$ is non-constant and bounded away from zero remains unknown.

 It is clear that $\Omega$ satisfies the BGD condition with $\mathcal A=\{1\}$ containing only one element and $\Gamma=\{\eta_1,\eta_2\}$ consisting of two directed edges from $\Omega$ to itself. Then $\nu(\Omega)=\nu(K)=1$ and the $1\times 1$ matrix $A$ is $2$, hence we are in the irreducible case with $d=\frac{2\log2}{\log5}$. By applying Theorem \ref{thm-irreducible}, we have for $*=D$ or $N$, there exists a bounded $\frac{\log 5}2$-periodic function $G_*$ such that as $x\rightarrow+\infty$,
\begin{equation}\label{rhostarsgdomain}
\rho^{\Omega}_*(x)=G\Big(\frac{\log x}{2}\Big)x^{\frac{\log3}{\log5}}+G_{*}\Big(\frac{\log x}{2}\Big) x^{\frac{\log2}{\log5}}+O(1),
\end{equation}
where we improve the term $o\Big(x^{\frac{\log2}{\log5}}\Big)$ in Theorem \ref{thm-irreducible} to $O(1)$ using the same argument as in the proof of Theorem \ref{KigamiJFA}, noticing that the matrix $A$ is now simply a real number.

By the max-min formula for eigenvalues (e.g. \cite[formula (2.15)]{La}), we observe that $\rho_D^\Omega(x)\leq \rho_D(x)\leq \rho_N(x)\leq \rho_N^{\Omega}(x)$. Since both $\rho_D(x)$ and $\rho_N(x)$ have the asymptotic behavior $G\Big(\frac{\log x}{2}\Big)x^{\frac{\log 3}{\log 5}}+O(1)$ as $x\rightarrow+\infty$, it follows that
$G_D\leq0$ and $G_N\geq0$.

Let
\begin{equation*}
Z^{\Omega}_D(t)=\int_0^{+\infty}e^{-tx}d\rho^{\Omega}_{D}(x),\quad t>0,
\end{equation*}
denote the {\it spectral partition function} of the Dirichlet Laplacian on $\Omega$. By applying Kajino's result \cite[Theorem 3.19]{Ka2} (with $m=1$ and $X=\{1,2\}$), there exist three positive, bounded, $\frac{\log5}2$-periodic, continuous functions $\hat G$, $\hat G_D$ and $\hat G_0$ such that as $t\rightarrow0+$,
\begin{equation}\label{Kajinoresult}
Z^{\Omega}_D(t)=\hat G\Big(-\frac{\log t}{2}\Big)t^{-\frac{\log3}{\log 5}}-\hat G_D\Big(-\frac{\log t}{2}\Big)t^{-\frac{\log 2}{\log 5}}+\hat G_0\Big(-\frac{\log t}{2}\Big)+O\left(\exp\Big(-ct^{-\frac{\log 2}{\log 5-\log 2}}\Big)\right).
\end{equation}
We establish the following relations between $G$ and $\hat G$, and between $G_D$ and $\widetilde G_D$:
\begin{proposition}\label{thm6.1}
\begin{align}
\hat G(x)&=\int_0^{+\infty} G\Big(\frac{\log\xi}{2}+x\Big)\xi^{\frac{\log 3}{\log 5}}e^{-\xi}d\xi,\label{eqGid}\\
\hat G_D(x)&=-\int_0^{+\infty} G_D\Big(\frac{\log\xi}{2}+x\Big)\xi^{\frac{\log 2}{\log 5}}e^{-\xi}d\xi.\label{eqGDid}
\end{align}
In particular, $G_D$ is not identically zero.
\end{proposition}
\begin{proof}
Using integration by parts and noting that $\rho_D^\Omega(x)$ has polynomial growth, we obtain
\begin{equation*}
Z^{\Omega}_D(t)=t\int_0^{+\infty}e^{-tx}\rho^{\Omega}_{D}(x)dx.
\end{equation*}
Substituting \eqref{rhostarsgdomain} into the above integral, we have as $t\rightarrow0+$,
\begin{equation*}
Z^{\Omega}_D(t)
=t\int_0^{+\infty}e^{-tx}G\Big(\frac{\log x}{2}\Big)x^{\frac{\log3}{\log5}}dx+t\int_0^{+\infty}e^{-tx}G_D\Big(\frac{\log x}{2}\Big)x^{\frac{\log2}{\log5}}dx+O(1).
\end{equation*}
This simplifies to
\begin{equation*}
Z^{\Omega}_D(t)=t^{-\frac{\log3}{\log5}}\int_0^{+\infty} G\Big(\frac{\log\xi}{2}-\frac{\log t}{2}\Big)\xi^{\frac{\log 3}{\log 5}}e^{-\xi}d\xi+t^{-\frac{\log2}{\log5}}\int_0^{+\infty} G_D\Big(\frac{\log\xi}{2}-\frac{\log t}{2}\Big)\xi^{\frac{\log 2}{\log 5}}e^{-\xi}d\xi+O(1).
\end{equation*}
By comparing this with \eqref{Kajinoresult}, we deduce \eqref{eqGid} and \eqref{eqGDid}. Moreover, since $\hat G_D$ is positive, it follows that $G_D$ is not identically zero.
\end{proof}

{\bf2.} $p=p_3$, $q=\frac12(p_1+p_2)$. Let $\Omega_1=H\cap K$, where $H$ is the half-plane containing $p_1$ with boundary line passing through $p$ and $q$,  and $\Omega_2=K\setminus\{p_2\}$. Then $\{\Omega_1,\Omega_2\}$ are open sets satisfying the BGD condition with $\mathcal A=\{1,2\}$ and $\Gamma=\{\eta_1,\eta_2,\eta_3\}$, where $\eta_1$ is from $1$ to $1$, $\eta_2$ is from $1$ to $2$ and $\eta_3$ is from $2$ to $2$, see Figure \ref{figsg2}.
\begin{figure}[h]
	\includegraphics[width=4.5cm]{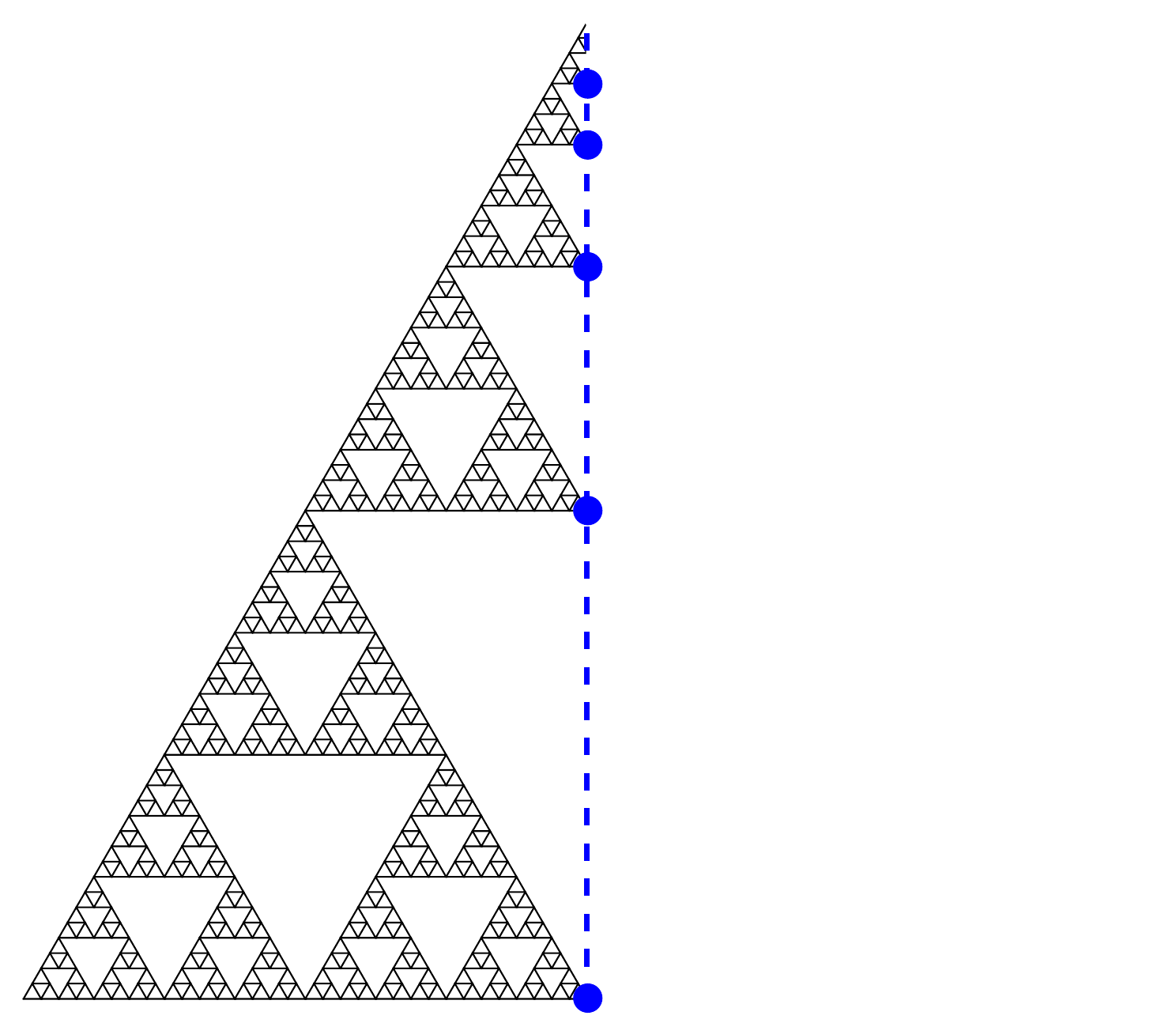}\hspace{1cm}
\includegraphics[width=4.5cm]{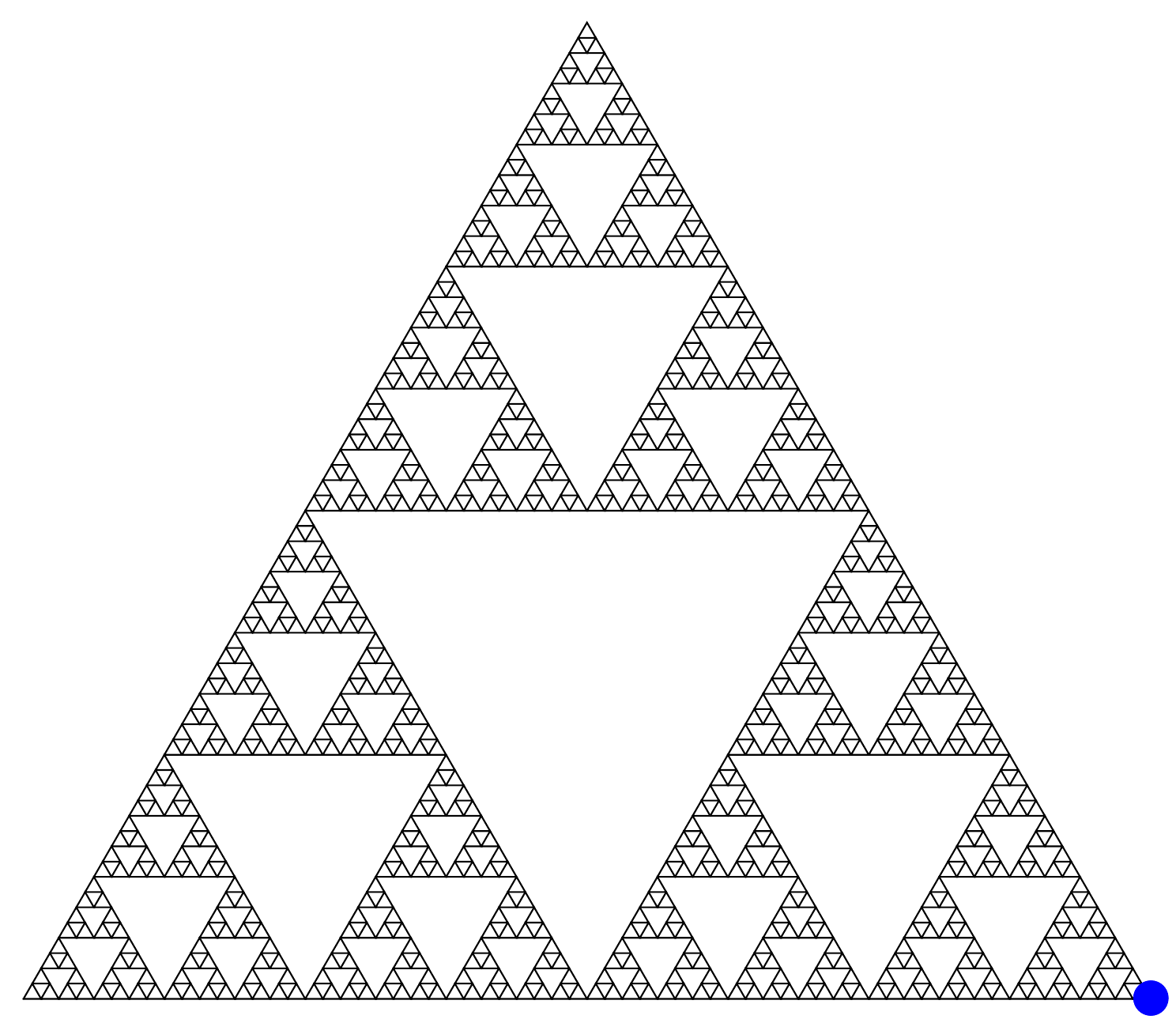}
	\caption{$\Omega_1$ and $\Omega_2$ in Example \ref{subsec6.1}-2}\label{figsg2}
\end{figure}

 Note that $A=\left(
                                                                  \begin{array}{cc}
                                                                    1 & 1 \\
                                                                    0 & 1 \\
                                                                  \end{array}
                                                                \right)
$ is reducible with spectral radius $1$, and $\{1\}$, $\{2\}$ are basic classes in $\mathcal A$ with height $1$, $0$, respectively. By Theorem \ref{reducible}-(3), we have for $*=D$ or $N$, as $x\rightarrow+\infty$,
\begin{align*}
\rho^{\Omega_1}_*(x)&=\frac12G\Big(\frac{\log x}{2}\Big)x^{\frac{\log3}{\log5}}+O(\log x),\\
\rho^{\Omega_2}_*(x)&=G\Big(\frac{\log x}{2}\Big)x^{\frac{\log3}{\log5}}+O(1).
\end{align*}

We remark that the above estimate for $\rho^{\Omega_1}_*(x)$ is sharp due to the following two formulas of Li and Strichartz \cite[Section 5]{LS}, derived via a symmetric spectral decimation argument: there exists $C_0>0$ such that
\begin{align*}
\rho^{\Omega_1}_N(C_05^m)&=\frac12\Big(\frac{3^{m+1}+3}{2}+m+1\Big),\\
\rho^{\Omega_1}_D(C_05^m)&=\frac12\Big(\frac{3^{m+1}-3}{2}-m\Big).
\end{align*}

\medskip

The following are three more examples of BGD open sets in the Sierpi{\'n}ski gasket.

\medskip
\begin{figure}[h]
	\includegraphics[width=6cm]{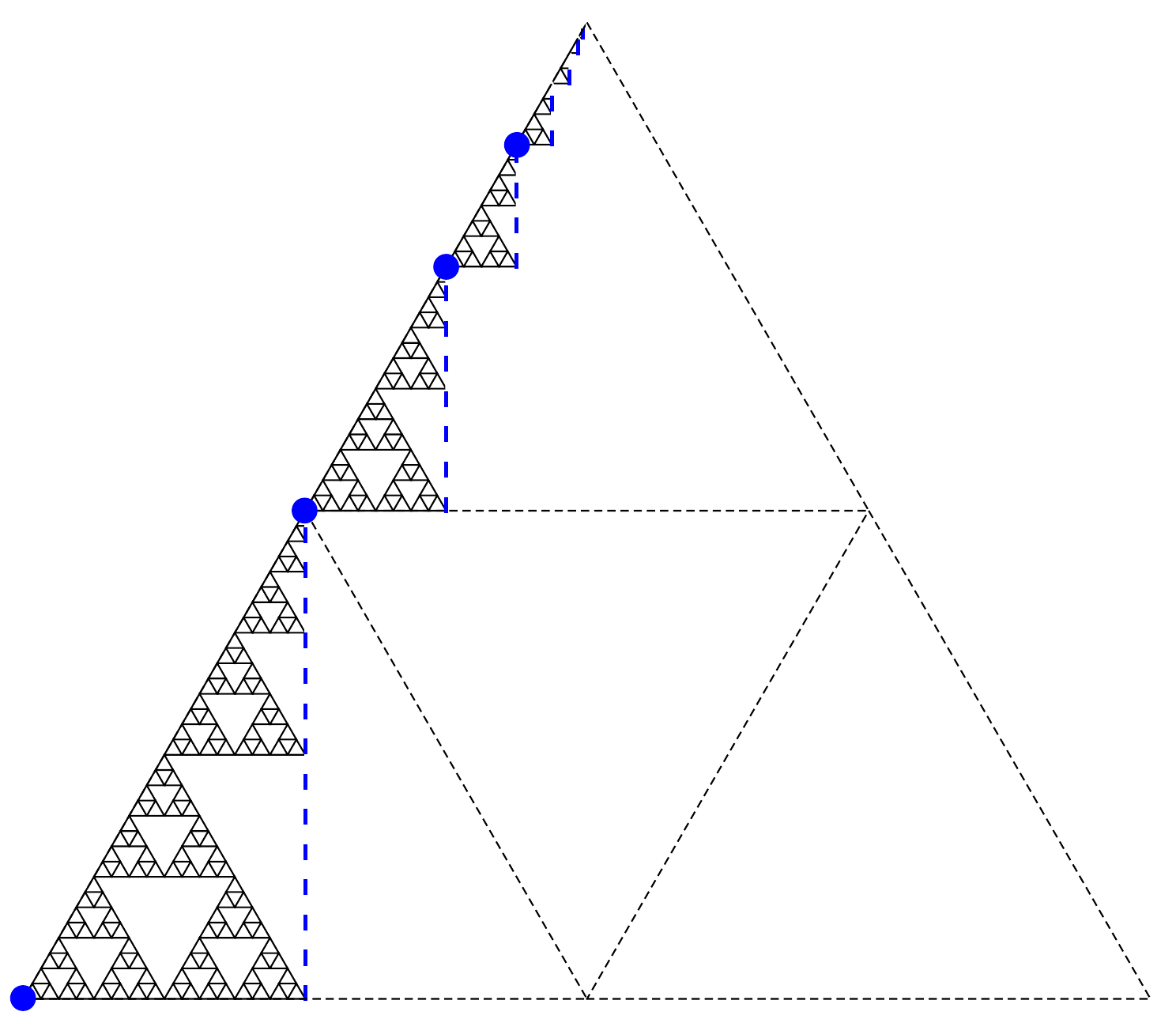}
	\caption{$\Omega_3$ in Example \ref{subsec6.1}-3}\label{figsg3}
\end{figure}
{\bf3.} Based on the example in {\bf 2}, let us consider the following open set $\Omega_3$, satisfying that $\Omega_3=F_3(\Omega_3)\cup F_1(\Omega_1\setminus\{p_1,p_3\})$, see Figure \ref{figsg3}.

 It is direct to check that $\Omega_3$ is a BGD open set with height $2$. By Theorem \ref{reducible}-(3), we have for $*=D$ or $N$, as $x\rightarrow+\infty$,
\begin{equation*}
\rho^{\Omega_3}_*(x)=\frac14G\Big(\frac{\log x}{2}\Big)x^{\frac{\log3}{\log5}}+O\Big((\log x)^2\Big).
\end{equation*}

\medskip

{\bf4.} For $\delta\in(0,1)$, Consider a horizontal line which intersects $\overline{p_1p_3}$ at a point with distance $\delta$ to $p_3$, and denote the open set of $K$ above this line as $\Omega_{\delta}$. Let us consider $\Omega_{2/3}$ and $\Omega_{1/3}$, see Figure \ref{figsg4}.
\begin{figure}[h]
	\includegraphics[width=4.5cm]{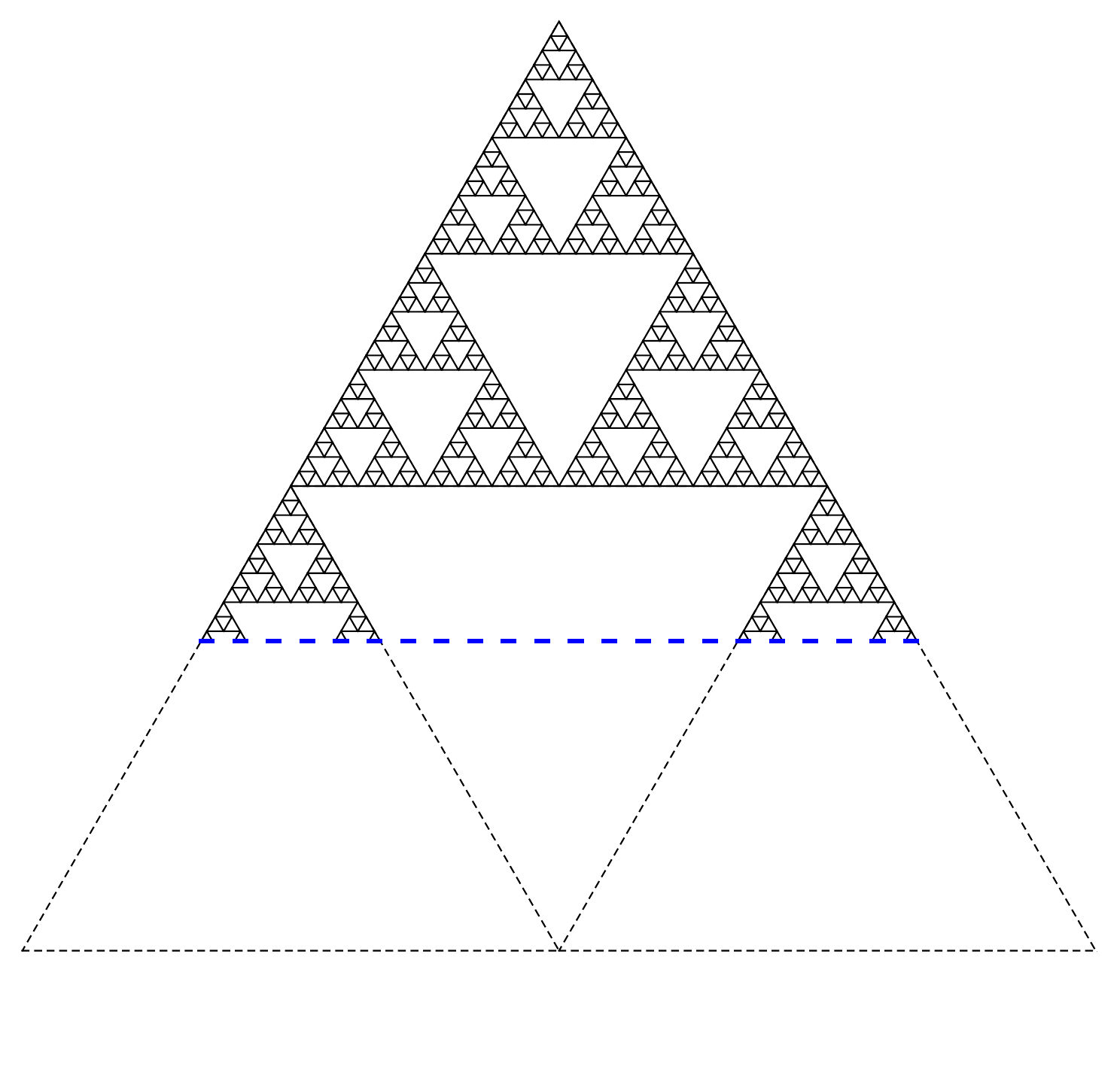}\hspace{1cm}
\includegraphics[width=4.5cm]{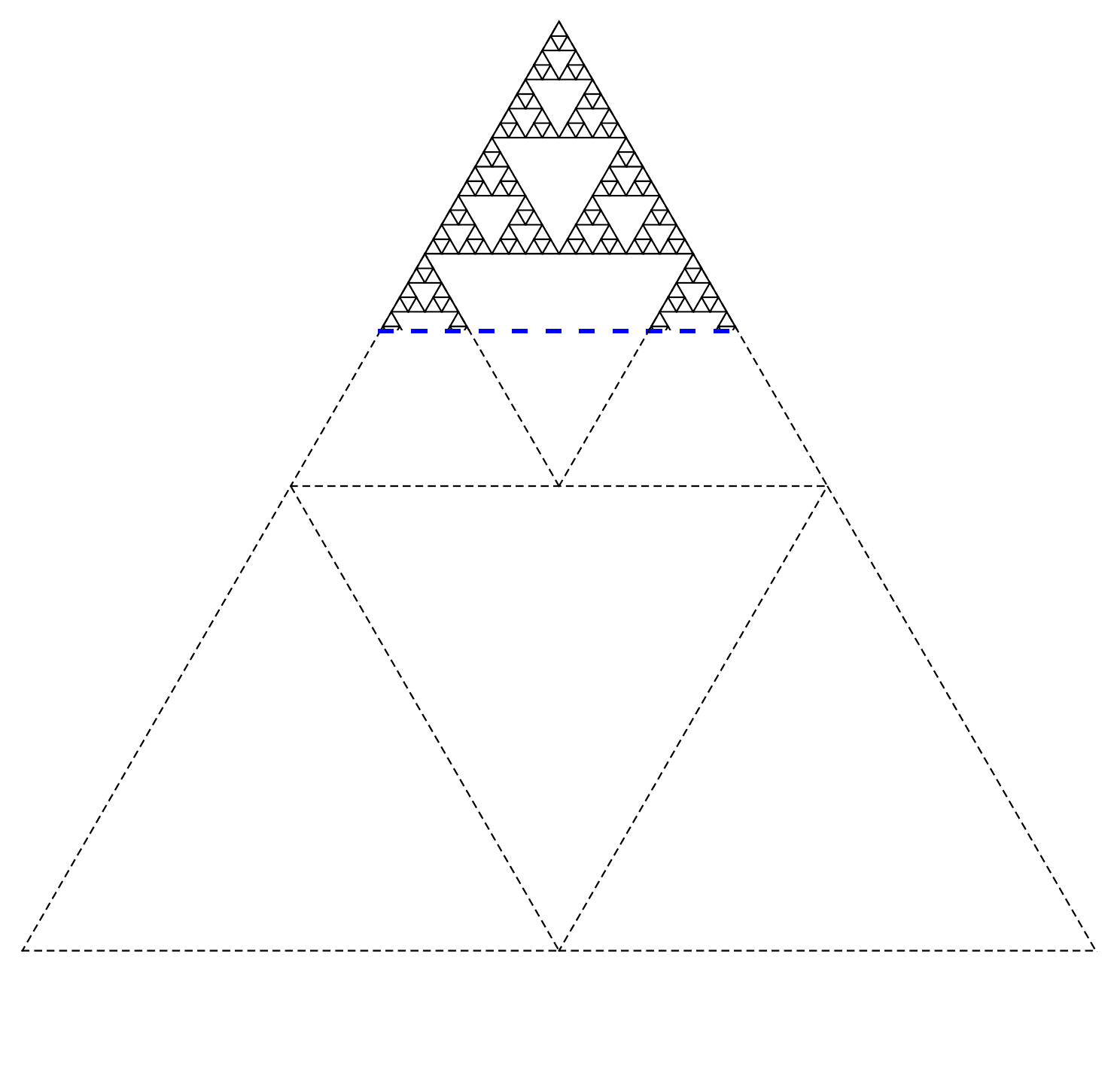}
	\caption{$\Omega_{2/3}$ and $\Omega_{1/3}$ in Example \ref{subsec6.1}-4}\label{figsg4}
\end{figure}   Then $\{\Omega_{2/3}, \Omega_{1/3}\}$ satisfies the BGD condition with an irreducible matrix $A=\Big(
                                                                                                             \begin{array}{cc}
                                                                                                               0 & 2 \\
                                                                                                               1 & 0 \\
                                                                                                             \end{array}
                                                                                                           \Big)
$ and $\Psi(A)=\sqrt 2$. By Theorem \ref{thm-irreducible}, we find that there exists a $\log 5$-periodic bounded function $G_*$ for $*=D$ or $N$, such that as $x\rightarrow+\infty$,
\begin{align*}
\rho^{\Omega_{2/3}}_*(x)&=\frac37G\Big(\frac{\log x}{2}\Big)x^{\frac{\log3}{\log5}}+\sqrt{2}G_*\Big(\frac{\log x}{2}\Big)x^{\frac{\log(\sqrt{2})}{\log 5}}+o\Big(x^{\frac{\log(\sqrt{2})}{\log 5}}\Big),\\
\rho^{\Omega_{1/3}}_*(x)&=\frac17G\Big(\frac{\log x}{2}\Big)x^{\frac{\log3}{\log5}}+G_*\Big(\frac{\log x}{2}-\frac{\log 5}{2}\Big)x^{\frac{\log(\sqrt{2})}{\log 5}}+o\Big(x^{\frac{\log(\sqrt{2})}{\log 5}}\Big).
\end{align*}

\medskip

{\bf 5.} Take $\widetilde\Omega$ to be the open set satisfying $\widetilde\Omega=F_1(\widetilde\Omega)\cup F_2(\widetilde\Omega)\cup F_{33}(\Omega\setminus\{p_3\})$, where $\Omega=K\setminus \overline{p_1p_2}$ as the example in {\bf 1}, see Figure \ref{figsg5}.
\begin{figure}[h]
	\includegraphics[width=6cm]{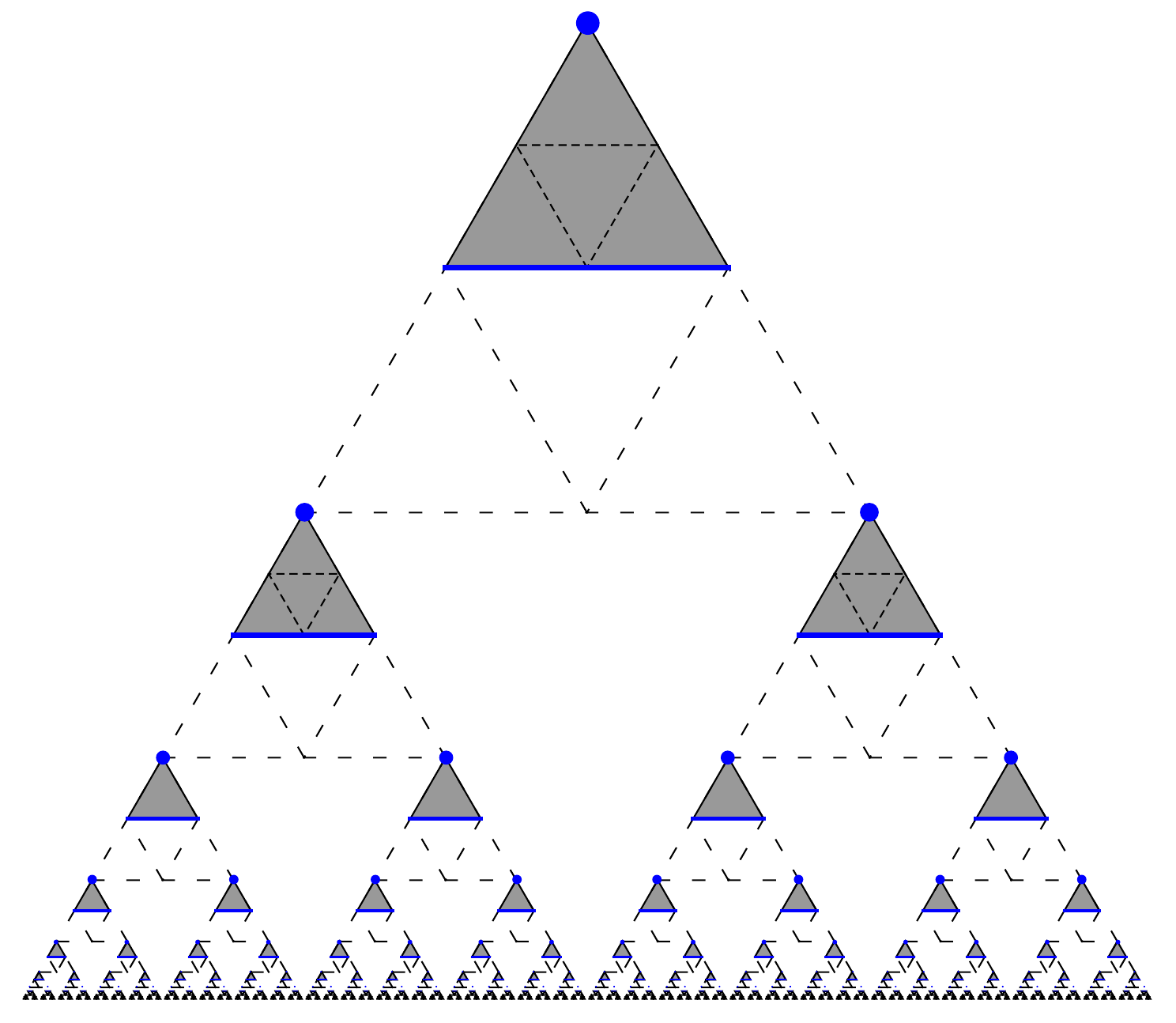}
	\caption{$\widetilde\Omega$ in Example \ref{subsec6.1}-5}\label{figsg5}
\end{figure}
Then $\widetilde\Omega$ satisfies the $\rm\widetilde{BGD}$ condition (see Remark 2 in Section \ref{sec3}) with spectral radius $2$ and height $1$. By Theorem \ref{reducible}-(1), there exists a $\frac{\log 5}2$-periodic bounded function $\widetilde G_*$ for $*=D$ or $N$, such that as $x\rightarrow+\infty$,
\begin{equation*}
\rho^{\widetilde\Omega}_*(x)=\frac13G\Big(\frac{\log x}{2}\Big)x^{\frac{\log3}{\log5}}+{\widetilde G}_{*}\Big(\frac{\log x}{2}\Big) x^{\frac{\log2}{\log5}}\log x+o\Big(x^{\frac{\log2}{\log5}}\log x\Big).
\end{equation*}

Indeed, since $\widetilde \Omega$ is a disjoint union of copies of $\Omega\setminus\{p_3\}$, an argument analogous to Lemma \ref{lemmaKL} yields
\begin{equation*}
{\widetilde G}_{*}=\frac1{4\log 5}{G}_{*},
\end{equation*}
where $G_*$ is the function defined in \eqref{rhostarsgdomain}.

\medskip

Now we turn to consider the general open subsets in $K$ which are not necessarily of BGD type.

\medskip

{\bf 6.} Let  $\Omega\subset K$ be a non-empty open set in $K$ whose boundary $D$ is non-empty and has the upper Minkowski dimension $\alpha_M\in[0,\frac{\log3}{\log2})$.
Let $\Sigma=\{1,2,3\}$. For $n\geq1$ and a word $\omega=\omega_1\cdots\omega_n\in\Sigma^n$, denote $\omega^-=\omega_1\ldots\omega_{n-1}$.
For $k\geq1$, define $\Lambda_k=\{\omega\in \Sigma^k:\ F_\omega(K)\subset \Omega, F_{\omega^-}(K)\nsubset\Omega\}$, then $\{F_{\omega}(K):\ \omega\in\Lambda_k,{k\geq1}\}$ forms a Whitney-type decomposition of $\Omega$. Let $\nu$ be the normalized $\frac{\log3}{\log2}$-dimensional Hausdorff measure on $K$. Then the measure of $\Omega$ is given by
\begin{equation*}
\nu(\Omega)=\sum_{k=1}^\infty\frac{\#\Lambda_k}{3^k}.
\end{equation*}
Define $\widetilde\Lambda_k=\{\omega\in \Sigma^k:\ F_\omega(K)\cap D\neq\emptyset\}$. Clearly, since $\#\Lambda_k\leq 3\#\widetilde\Lambda_k$ for $k\geq1$, we see that
\begin{equation*}
\alpha_I:=\limsup\limits_{k\rightarrow+\infty}\frac{\log(\#\Lambda_k)}{k\log2}\leq \limsup\limits_{k\rightarrow+\infty}\frac{\log(\#\widetilde\Lambda_k)}{k\log2}=\alpha_M.
\end{equation*}
\begin{proposition}
For any $\varepsilon\in(0,\frac{\log3}{\log2}-\alpha_M)$, there exists $C>0$ such that for $x>0$,
\begin{equation*}
\nu(\Omega)G\Big(\frac{\log x}{2}\Big)x^{\frac{\log 3}{\log 5}}-Cx^{\frac{(\alpha_I+\varepsilon)\log2}{\log5}}\leq \rho_D^\Omega(x)\leq \nu(\Omega)G\Big(\frac{\log x}{2}\Big)x^{\frac{\log 3}{\log 5}}+Cx^{\frac{(\alpha_M+\varepsilon)\log2}{\log5}}.
\end{equation*}
\end{proposition}
\begin{proof}
The proof is inspired by \cite[Theorem 2.1]{La} of Lapidus, see also \cite[Proposition 12.6]{Fa}. In the following, we use $C$ to denote a positive constant which may vary in value.

First, let us look at the lower bound. Let $\Omega_n=\bigcup_{k=1}^n\bigcup_{\omega\in\Lambda_k}F_{\omega}(K)$ be the $n$-th approximation of $\Omega$. Then $\Omega_n\subset\Omega$, and we have $\rho_D^\Omega(x)\geq \rho_D^{\Omega_n}(x)$. By putting Dirichlet boundary condition on each cell $F_\omega(K)$ in the above decomposition of $\Omega_n$, we see that
\begin{equation}\label{eqqeSGe1}
\rho_D^{\Omega_n}(x)\geq \sum_{k=1}^n\sum_{\omega\in\Lambda_k}\rho_D^{F_{\omega}(K\setminus V_0)}(x).
\end{equation}
By \eqref{Kigamiformula}, there is a constant $C>0$ (independent of $x$ and of $k$) such that
\begin{equation}\label{eqqeSGe2}
\rho_D^{F_{\omega}(K\setminus V_0)}(x)\geq\frac{1}{3^k}G\Big(\frac{\log x}{2}\Big)x^{\frac{\log3}{\log5}}-C.
\end{equation}
Substituting \eqref{eqqeSGe2} into \eqref{eqqeSGe1}, we obtain
\begin{align*}
\rho_D^{\Omega_n}(x)&\geq\sum_{k=1}^n\#\Lambda_k\left(\frac{1}{3^k}G\Big(\frac{\log x}{2}\Big)x^{\frac{\log3}{\log5}}-C\right)\\
&=\Big(\nu(\Omega)-\sum_{k=n+1}^\infty\frac{\#\Lambda_k}{3^k}\Big)G\Big(\frac{\log x}{2}\Big)x^{\frac{\log3}{\log5}}-C\sum_{k=1}^n\#\Lambda_k\\
&\geq\nu(\Omega)G\Big(\frac{\log x}{2}\Big)x^{\frac{\log3}{\log5}}-C\Big(\sum_{k=n+1}^\infty\frac{\#\Lambda_k}{3^k}x^{\frac{\log3}{\log5}}+\sum_{k=1}^n\#\Lambda_k\Big).
\end{align*}
By the definition of $\alpha_I$, we have for any $\varepsilon\in(0,\frac{\log3}{\log2}-\alpha_I)$, there exists $k_0$ sufficiently large such that for all $k\geq k_0$,
$\log(\#\Lambda_k)\leq (\alpha_I+\varepsilon)k\log 2$; while for $k< k_0$, we simply have $\#\Lambda_k\leq 3^k\nu(\Omega)< 3^{k_0}\nu(\Omega)$. Hence, we get
\begin{align*}
\rho_D^{\Omega}(x)&\geq\nu(\Omega)G\Big(\frac{\log x}{2}\Big)x^{\frac{\log3}{\log5}}-C\Big(\sum_{k=n+1}^\infty\frac{2^{(\alpha_I+\varepsilon)k}}{3^k}x^{\frac{\log3}{\log5}}+\sum_{k=1}^n2^{(\alpha_I+\varepsilon)k}\Big)\\
&\geq\nu(\Omega)G\Big(\frac{\log x}{2}\Big)x^{\frac{\log3}{\log5}}-C\Big(\frac{1}{3^n}x^{\frac{\log3}{\log5}}+1\Big)2^{(\alpha_I+\varepsilon)n}.
\end{align*}
Taking $n$ to be the smallest positive integer $n\geq\frac{\log x}{\log 5}$, we see that
\begin{equation*}
\rho_D^{\Omega}(x)\geq\nu(\Omega)G\Big(\frac{\log x}{2}\Big)x^{\frac{\log3}{\log5}}-Cx^{\frac{(\alpha_I+\varepsilon)\log2}{\log5}},
\end{equation*}
which is the desired estimate.

The argument for the upper bound is quite similar.

Denote $\widetilde\Omega_n=\Omega_n\bigcup\Big(\bigcup_{\omega\in\widetilde\Lambda_n}F_{\omega}(K)\Big)$. Noting that $\Omega\subset\widetilde\Omega_n$, we have
\begin{align*}
\rho_D^\Omega(x)&\leq \rho_D^{\widetilde\Omega_n}(x)\leq\sum_{\omega\in \bigcup_{k=1}^n \Lambda_k\bigcup\widetilde\Lambda_n}\rho^{F_{\omega}(K\setminus V_0)}_D(x)\notag+C\Big(\sum_{k=1}^n{\#\Lambda_k}+\#\widetilde\Lambda_n\Big)\\
&\leq\sum_{k=1}^n\frac{\#\Lambda_k}{3^k}G\Big(\frac{\log x}{2}\Big)x^{\frac{\log3}{\log5}}+\frac{\#\widetilde\Lambda_n}{3^n}G\Big(\frac{\log x}{2}\Big)x^{\frac{\log3}{\log5}}+C\Big(\sum_{k=1}^n{\#\Lambda_k}+\#\widetilde\Lambda_n\Big)\\
&\leq\nu(\Omega)G\Big(\frac{\log x}{2}\Big)x^{\frac{\log3}{\log5}}+C\Big(\frac{\#\widetilde\Lambda_n}{3^n}x^{\frac{\log3}{\log5}}+\sum_{k=1}^n{\#\widetilde\Lambda_k}\Big).
\end{align*}
By the definition of $\alpha_M$, for any $\varepsilon\in(0,\frac{\log3}{\log2}-\alpha_M)$, there exists $n_0\geq1$ large enough such that for any $n\geq n_0$,
$\log(\#\widetilde\Lambda_n)\leq (\alpha_M+\varepsilon)n\log 2$. Taking $n$ to be the smallest positive integer such that $n\geq \frac{\log x}{\log 5}$, we arrive at
\begin{equation*}
\rho_D^\Omega(x)\leq\nu(\Omega)G\Big(\frac{\log x}{2}\Big)x^{\frac{\log3}{\log5}}+Cx^{\frac{(\alpha_M+\varepsilon)\log2}{\log5}},
\end{equation*}
proving the upper bound.
\end{proof}

\noindent{\bf Remark.} {\it If $\Omega$ satisfies the BGD condition with a boundary $D$, the exponents $\alpha_I$ and $\alpha_M$ are equal, and they are also equal to the Hausdorff dimension $\alpha_{D}\in[0,\frac{\log3}{\log2})$ of $D$. According to Theorem \ref{reducible}, we actually have a finer estimate of the second term, i.e. $G_D\Big(\frac{\log x}{2}\Big)x^{\frac{\alpha_{D}}{\beta}}(\log x)^m$ ($\alpha_{D}>0$) or $O\Big((\log x)^m\Big)$ ($\alpha_D=0$), where $\beta=\frac{\log5}{\log2}$ is the walk dimension of $K$.}

\subsection{Example: Lindstr{\o}m snowflake}\label{subsec6.2} Let $\left\{p_k=\exp\Big(\frac{2k\pi }{6}{\rm i}\Big)\right\}_{k=1}^6$ represent the six vertices of a regular hexagon, and $p_7=0$. For $1\leq k\leq7$, define $F_k$ to be the similitude on the plane give by $F_k(x)=\frac13(x-p_k)+p_k$. The self-similar set $K$ generated by the IFS $\{F_k\}_{k=1}^7$ is a p.c.f. self-similar set and typically a nested fractal, called the {\it Lindstr{\o}m snowflake}. Let $\mu$ be the $\frac{\log 7}{\log3}$-dimensional Hausdorff measure on $K$. By Lindsr{\o}m's theorem for nested fractals \cite{L}, there exists a self-similar Dirichlet form $(\mathcal E,\mathcal F)$ in $L^2(K,\mu)$ with a common energy renormalizing factor $r\in(0,1)$.
Consider the BGD domains $\{\Omega_1,\Omega_2,\Omega_3\}$ with Koch curve boundaries as introduced in Section \ref{sec 1}, see Figure \ref{figureSF}. The incidence matrix is
\begin{equation*}
A=\left(
    \begin{array}{ccc}
      0 & 2 & 0 \\
      0 & 2 & 1 \\
      0 & 2 & 3 \\
    \end{array}
  \right),
\end{equation*} satisfying $\Psi(A)=4$. Denote $G$ to be the $\frac{\log(7r^{-1})}{2}$-periodic function in Theorem \ref{thm-KL}. Then
$\{2,3\}$ forms a basic class, so we apply Theorem \ref{thm-irreducible} to see that: for $*=D$ or $N$, there exists a $\frac{\log(7r^{-1})}{2}$-periodic bounded function $G_*$ such that as $x\rightarrow+\infty$,
\begin{align*}
\rho^{\Omega_2}_*(x)&=G\Big(\frac{\log x}{2}\Big)x^{\frac{\log7}{\log(7r^{-1})}}+G_*\Big(\frac{\log x}{2}\Big)x^{\frac{\log4}{\log(7r^{-1})}}+o\Big(x^{\frac{\log4}{\log(7r^{-1})}}\Big),\\
\rho^{\Omega_3}_*(x)&=G\Big(\frac{\log x}{2}\Big)x^{\frac{\log7}{\log(7r^{-1})}}+2G_*\Big(\frac{\log x}{2}\Big)x^{\frac{\log4}{\log(7r^{-1})}}+o\Big(x^{\frac{\log4}{\log(7r^{-1})}}\Big).
\end{align*}
Also by the relation between $\Omega_1$ and $\Omega_2$, using Corollary \ref{thmcor}, we further see that as $x\rightarrow+\infty$,
\begin{equation*}
\rho^{\Omega_1}_*(x)=G\Big(\frac{\log x}{2}\Big)x^{\frac{\log7}{\log(7r^{-1})}}+\frac12G_*\Big(\frac{\log x}{2}\Big)x^{\frac{\log4}{\log(7r^{-1})}}+o\Big(x^{\frac{\log4}{\log(7r^{-1})}}\Big).
\end{equation*}
Note that in the above, the numbers $1,2,\frac12$ appeared as coefficients in the second order term represent ratios of the $\frac{\log 4}{\log 3}$-dimensional Hausdorff measures of the boundaries of $\Omega_1,\Omega_2,\Omega_3$.

\section{Appendix: Vector-valued renewal theorems}\label{sec 7}
 \setcounter{equation}{0}\setcounter{theorem}{0}

In this section, we present the vector-valued renewal theorems established by Lau, Wang, and Chu \cite[Theorems 4.2, 4.3 (for irreducible case), Theorem 4.5 (for general case)]{LWC}, and also refer to Hambly and Nyberg \cite[Theorems 2.1, 2.2, 2.6]{HN}. These results are precisely what we apply to derive the Weyl-Berry spectral asymptotics.

For a Radon measure $\mu$ on $\mathbb R$, denote $\mu(x)=\mu(-\infty,x]$ for $x\in\mathbb R$, and $\mu(x,x+h]=\mu(x+h)-\mu(x)$ for $h>0$.
Let $U=(U_{ij})_{n\times n}$ be a matrix of finite Radon measures defined on $\mathbb R$ vanishing on $(-\infty,0)$. Denote $U(+\infty)=\big(U_{ij}(+\infty)\big)_{n\times n}$ be the matrix of the total variations of the measures. Let ${M}=\Big(\int_{0}^{+\infty}xdU_{ij}(x)\Big)_{n\times n}=:(m_{ij})_{n\times n}$ be the first moment matrix.

By viewing $\{1,\ldots,n\}$ as the {\it state space}, we use ${\bf\eta}=(i_1,\ldots,i_k)$ to denote the {\it path} starting from state $i_1$ and visiting $i_2,\ldots,i_k$ successively. Such a path $\bf\eta$ is called a {\it cycle} if $i_1=i_k$, and a {\it simple cycle} if it is a cycle and all $i_1,\ldots,i_{k-1}$ are distinct. For a path ${\bf\eta}=(i_1,\ldots,i_k)$, we denote $U_{\bf\eta}=U_{i_1i_2}*\cdots*U_{i_{k-1}i_k}$.
Define $\mathcal G_U$ to be the closed subgroup of $\mathbb R$ generated by
\begin{equation*}
\bigcup\Big\{\text{supp }U_{\bf\eta}:\ {\bf\eta} \text{ is a simple cycle on $\{1,\ldots,n\}$}\Big\}.
\end{equation*}

\begin{theorem}\label{thmLWC1}[Lau-Wang-Chu] Suppose $U$ is a matrix of finite Radon measures defined on $\mathbb R$ such that each non-zero entry is non-degenerate at $0$ ($\text{\rm supp } U_{ij}\neq\{0\}$ providing $U_{ij}\neq 0$) and vanishes on $(-\infty,0)$. Also, suppose $U(+\infty)$ is irreducible and has maximal eigenvalue $1$. Let $W=\sum_{k=0}^{+\infty} U^{*k}$.

1. Non-lattice case: if $\mathcal G_U=\mathbb R$, then for any $h>0$,
\begin{equation}\label{eqdensegroup}
\lim_{x\rightarrow+\infty}W(x,x+h]=hB,
\end{equation}
where
\begin{equation*}
B=\frac1\alpha {\bf u}{\bf v}^T,\quad  \alpha={\bf v}^TM{\bf u},
\end{equation*}
($B=0$ if one of the $m_{ij}$ is $+\infty$) and ${\bf u},{\bf v}$ are the unique normalized right and left $1$-eigenvectors of $U(+\infty)$ respectively.

2. Lattice case: if $\mathcal G_U=\langle\varrho\rangle$ for some $\varrho>0$, then for any $a_{ij}\in \text{\rm supp } U_{{\bf\eta}(i,j)}$,
\begin{equation}\label{eqdiscretgroup}
\lim_{x\rightarrow+\infty}\Big(W_{ij}(x+a_{ij},x+a_{ij}+\varrho]\Big)_{n\times n}=\varrho B.
\end{equation}
\end{theorem}
\noindent{\bf Remark.} {\it In the above theorem, $\langle\varrho\rangle$ means the subgroup of $\mathbb R$ generated by $\varrho$, ${\bf\eta}(i,j)$ is any path from $i$ to $j$ such that $U_{{\bf\eta}(i,j)}\neq0$.}

\medskip

By using Theorem \ref{thmLWC1}, the following vector-valued renewal theorem holds, see \cite[Theorem 4.3]{LWC}. Say a function
$f:\ \mathbb R\rightarrow\mathbb R$ {\it directly Riemann integrable} if it is Riemann integrable on any finite interval and $\sum_{k\in\mathbb Z}\sup_{x\in(k,k+1]}|f(x)|<+\infty$.

\begin{theorem}\label{thmLWC2}[Lau-Wang-Chu] Under the same hypotheses on $U$ as in Theorem \ref{thmLWC1}, let ${\bf z}=(z_1,\ldots,z_n)^T$ be a vector of directly Riemann integrable functions with ${\bf z}(x)={\bf0}$ for $x<0$. Then ${\bf f}(x)=W*{\bf z}(x)$ is a bounded solution of
\begin{equation*}
{\bf f}(x)={\bf z}(x)+U*{\bf f}(x),\quad x\geq0,
\end{equation*}
and it is unique in the class of bounded functions that vanish on $(-\infty,0)$. Furthermore, if $\mathcal G_U=\mathbb R$, then
\begin{equation*}
\lim_{x\rightarrow+\infty}{\bf f}(x)=B\Big(\int_0^{+\infty}{\bf z}(t)dt\Big),
\end{equation*}
where $B$ is defined in Theorem \ref{thmLWC1}. If $\mathcal G_U=\langle\varrho\rangle$ for some $\varrho>0$, then for $a_{j1}\in\text{\rm supp }U_{{\bf\eta}(j,1)}$,
\begin{equation}\label{eqeqconv}
\lim_{x\rightarrow+\infty}\left(\left(
                                                                           \begin{array}{c}
                                                                            f_1(x+a_{11}) \\
                                                                             \ldots\\
                                                                            f_n(x+a_{n1})\\
                                                                           \end{array}
                                                                         \right)-\varrho B\sum_{k\in\mathbb Z}\left(
                                                                           \begin{array}{c}
                                                                         z_1(x+a_{11}+k\varrho) \\
                                                                             \ldots\\
                                                                           z_n(x+a_{n1}+k\varrho)\\
                                                                           \end{array}
                                                                         \right)\right)=0.
\end{equation}
\end{theorem}

For our usage, we need the lattice case \eqref{eqeqconv}. Here is a short proof for this by using \eqref{eqdiscretgroup}. \begin{proof}[Proof of \eqref{eqeqconv}]For $i=1,\ldots,n$, by that
${\bf f}(x)=W*{\bf z}(x)$, we have (by denoting $B=(B_{ij})_{n\times n}$)
\begin{align}
&f_{i}(x+a_{i1})-\varrho \sum_{j=1}^nB_{ij}\sum_{k\in\mathbb Z}z_j(x+a_{j1}+k\varrho)\notag\\
=&\sum_{j=1}^n\Big(\int_0^{+\infty} z_j(t)dW_{ij}(x+a_{i1}-t)-\varrho B_{ij}\sum_{k\in\mathbb Z}z_j(x+a_{j1}+k\varrho)\Big).\label{eqeeq1}
\end{align}
 For each $j=1,\dots, n$ and $\varepsilon>0$, by that $z_j$ is directly Riemann integrable, there exists $N>0$ such that $\big|\int_{N}^{+\infty} z_j(t)dW_{ij}(x+a_{i1}-t)\big|<\varepsilon$. Using \eqref{eqdiscretgroup}, we see that
\begin{equation}\label{eqeeq2}
\lim_{x\rightarrow+\infty}\Big(\int_{0}^Nz_j(t)dW_{ij}(x+a_{i1}-t)-\varrho B_{ij}\sum_{k\in\mathbb Z:0\leq x+a_{i1}-a_{ij}+k\varrho\leq N}z_j(x+a_{i1}-a_{ij}+k\varrho)\Big)=0,
\end{equation}
while for $N>0$ sufficiently large, again by that $z_j$ is directly Riemann integrable,
\begin{equation}\label{eqeeq3}
\varrho B_{ij}\Big|\sum_{k\in\mathbb Z:0\leq x+a_{i1}-a_{ij}+k\varrho\leq N}z_j(x+a_{i1}-a_{ij}+k\varrho)-\sum_{k\in\mathbb Z}z_j(x+a_{i1}-a_{ij}+k\varrho)\Big|\leq\varepsilon.
\end{equation}
Combining \eqref{eqeeq2} and \eqref{eqeeq3}, together with $\sum_{k\in\mathbb Z}z_j(x+a_{i1}-a_{ij}+k\varrho)=\sum_{k\in\mathbb Z}z_j(x+a_{j1}+k\varrho)$, we obtain
\begin{equation*}
\lim_{x\rightarrow+\infty}\Big|\int_0^{+\infty} z_j(t)dW_{ij}(x+a_{i1}-t)-\varrho B_{ij}\sum_{k\in\mathbb Z}z_j(x+a_{j1}+k\varrho)\Big|\leq 2\varepsilon.
\end{equation*}
Letting $\varepsilon\rightarrow0$, substituting the above into \eqref{eqeeq1}, we see that
\begin{equation*}
\lim_{x\rightarrow+\infty}\Big(
f_{i}(x+a_{i1})-\varrho \sum_{j=1}^nB_{ij}\sum_{k\in\mathbb Z}z_j(x+a_{j1}+k\varrho)\Big)=0,
\end{equation*}
proving \eqref{eqeqconv}.
\end{proof}

Now, let $A$ be an irreducible non-negative $n\times n$ matrix. Define $t_{ij}=\min\ \{k\geq1:\ A^k(i,j)>0\}$ for $i,j=1,\ldots,n$. Let $\mathcal G_{i}$ be the subgroup of $\mathbb Z$ generated by $\{k\geq1:\ A^k(i,i)>0\}$, and $t_i\geq1$ be the generator of $\mathcal G_{i}$. Let $\varrho$ be the greatest common divisor of $t_1,\ldots,t_n$. Note that $\varrho$ is the generator of the subgroup in $\mathbb Z$ generated by $t_1,\ldots,t_n$. Let $\bf u$ and $\bf v$ be the normalized right and left $1$-eigenvectors of $A$. By applying Theorem \ref{thmLWC2} (the lattice case), we have the following corollary.

\begin{corollary}\label{thm7.3} Let $A$ be an irreducible non-negative $n\times n$ matrix with spectral radius $1$, and $\bf z$ be a vector of directly Riemann integrable functions on $\mathbb R$ with ${\bf z}(x)={\bf0}$ for $x<0$.
Then for $T>0$, ${\bf f}(x)=\sum_{k=0}^\infty A^k{\bf z}(x-kT)$ is a bounded solution of the equation
\begin{equation*}
{\bf f}(x)=A{\bf f}(x-T)+{\bf z}(x),\quad x\geq0,
\end{equation*}
and it is unique in the class of bounded functions that vanish on $(-\infty,0)$.
Moreover, $\bf f$ satisfies
\begin{equation*}
\lim_{x\rightarrow+\infty}\left(\left(
                                                                           \begin{array}{c}
                                                                             f_1(x+t_{11}T) \\
                                                                             \ldots\\
                                                                             f_n(x+t_{n1}T) \\
                                                                           \end{array}
                                                                         \right)
-\varrho B\sum_{k\in\mathbb Z}\left(
                                                                           \begin{array}{c}
                                                                            z_1(x+t_{11}T+k\varrho T) \\
                                                                             \ldots\\
                                                                            z_n(x+t_{n1}T+k\varrho T) \\
                                                                           \end{array}
                                                                         \right)\right)=0,
\end{equation*}
where $B=\frac1{T{\bf v}^T{\bf u}}{\bf u}{\bf v}^T$.
\end{corollary}

\begin{proof}
By letting $U={\delta_T}A$ and noting that $M=TA$  in Theorems \ref{thmLWC1}, \ref{thmLWC2}, the assertion is immediate.
\end{proof}

%
%
%
%

%

The following theorem is an extension of Theorem \ref{thmLWC2} from irreducible case to general case, which is
 due to Lau-Wang-Chu \cite[Theorem 4.5]{LWC} and Hambly-Nyberg \cite[Theorem 2.6]{HN}.
\begin{theorem}\label{thmHN}[Lau-Wang-Chu, Hambly-Nyberg]
Suppose $U$ is a matrix of finite Radon measures defined on $\mathbb R$ such that each non-zero entry is non-degenerate at $0$ and vanishes on $(-\infty,0)$. Also, suppose $U(+\infty)$  has maximal eigenvalue $1$, and each row has at least one non-zero entry. Assume $\int_0^{+\infty}xdU_{ij}(x)<+\infty$ for all $i,j$.  Let $W=\sum_{k=0}^{+\infty} U^{*k}$.  Let $\bf z$ be a vector of directly Riemann integrable functions on $\mathbb R$ with ${\bf z}(x)={\bf0}$ for $x<0$.
If $\bf f$ is bounded on finite intervals, vanishes on $(-\infty,0)$, and satisfies the renewal equation
\begin{equation*}
{\bf f}(x)={\bf z}(x)+U*{\bf f}(x),\quad x\geq0,
\end{equation*}
then ${\bf f}(x)=W*{\bf z}(x)$ and the components $f_i$ satisfy:

(1). if $j\rightarrow \mathcal S$, then
\begin{equation*}
\lim_{x\rightarrow+\infty}(x^{-{m_j}}f_j(x)-p_j(x))=0,
\end{equation*}
where $p_j$ is either a $\varrho_j$-periodic function or a constant  depending on whether $\mathcal G_U$ is lattice or not, and $m_j$, $\varrho_j$ are defined as in \eqref{defmj}, \eqref{defrhoj};

(2). if $j\nrightarrow \mathcal S$, then
\begin{equation*}
\lim_{x\rightarrow+\infty}f_j(x)=0.
\end{equation*}
\end{theorem}

\bibliographystyle{siam}

\bigskip


\end{document}